\documentclass[11pt,a4paper]{snkart}

\usepackage{tikz}
\usepackage{mathrsfs}
\usepackage{mathtools}

\let\rel\mathbf

\let\meet\wedge

\let\union\cup
\let\bd\partial
\let\impl\Rightarrow

\DeclareMathOperator{\ar}{ar}

\newcommand{\true}{\textsf{\upshape true}}
\newcommand{\false}{\textsf{\upshape false}}
\newcommand{\vertex}{{\operatorname{\sf vertex}}}
\newcommand{\edge}{\operatorname{\sf edge}}
\newcommand{\pr}{\operatorname{\sf pr}}

\newcommand{\e}{\bigtimes\nolimits}
\newcommand{\w}{\operatorname{\text{\normalfont\sf\&}}}

\newcommand{\relT}[1]{{\rel T(#1)}}
\newcommand{\relcheck}[1]{{\rel{\check #1}}}
\newcommand{\trees}{\mathscr T}
\newcommand{\subsets}[1]{{\mathscr P}(#1)}

\theoremstyle{remark}

\hypersetup{pdftitle  = {Functors on relational structures which admit both left and right adjoints},
  pdfauthor = {V. Dalmau, A. Krokhin, and J. Opršal}}

\begin{document}

\title[Functors which admit both left and right adjoints]{Functors on relational structures which admit both left and right adjoints}

\author{Víctor Dalmau}
\address{Universitat Pompeu Fabra, Barcelona, Spain}
\email{victor.dalmau@upf.edu}

\author{Andrei Krokhin}
\address{Durham University, UK}
\email{andrei.krokhin@durham.ac.uk}

\author{Jakub Opršal}
\address{University of Birmingham, UK}
\email{j.oprsal@bham.ac.uk}

\thanks{Victor Dalmau was supported by the MICIN under grants PID2019-109137GB-C22 and PID2022-138506NB-C22, and the Maria de Maeztu program (CEX2021-001195-M).
Andrei Krokhin and Jakub Opršal were supported by the UK EPSRC grant EP/R034516/1.
Andrei Krokhin was also supported by the UK EPSRC grant EP/X033201/1.
This project has received funding from the European Union’s Horizon 2020 research and innovation programme under the Marie Skłodowska-Curie Grant Agreement No 101034413.\\\medskip
\copyright V.~Dalmau, A.~Krokhin, J.~Opršal; licensed under Creative Commons License CC-BY 4.0.}

\begin{abstract}
  This paper describes several cases of adjunction in the homomorphism preorder of relational structures. We say that two functors $\Lambda$ and $\Gamma$ between thin categories of relational structures are \emph{adjoint} if for all structures $\mathbf A$ and $\mathbf B$, we have that $\Lambda(\mathbf A)$ maps homomorphically to $\mathbf B$ if and only if $\mathbf A$ maps homomorphically to $\Gamma(\mathbf B)$. If this is the case, $\Lambda$ is called the left adjoint to $\Gamma$ and $\Gamma$ the right adjoint to $\Lambda$.
  In 2015, Foniok and Tardif described some functors on the category of digraphs that allow both left and right adjoints. The main contribution of Foniok and Tardif is a construction of right adjoints to some of the functors identified as right adjoints by Pultr in 1970. We generalise results of Foniok and Tardif to arbitrary relational structures, and coincidently, we also provide more right adjoints on digraphs, and since these constructions are connected to finite duality, we also provide a new construction of duals to trees.
  Our results are inspired by an application in promise constraint satisfaction --- it has been shown that such functors can be used as efficient reductions between these problems.
\end{abstract}


\maketitle

\section{Introduction}

The study of homomorphisms between (di)graphs and general relational structures plays an important role in combinatorics and theoretical computer science \cite{HN04,FV98,KZ17book}. The importance in theoretical computer science stems, in particular, from the fact that the Constraint Satisfaction Problem (CSP) and its relatives can be cast as the problem of existence of a homomorphism from one relational structure to another. The class of all relational structures of a given signature (e.g., all digraphs) admits the homomorphism preorder, where $\rel A \le \rel B$ for two structure $\rel A$ and $\rel B$ if and only if there exists a homomorphism from $\rel A$ to $\rel B$. This preorder is interesting in its own right \cite{HN04}, and, moreover, important computational problems such as non-uniform CSPs and promise CSPs can be stated in terms of a relative position of an input structure with respect to a fixed structure or a pair of fixed structures. Specifically, the CSP with a template $\rel A$ (which is a fixed structure) asks whether a given structure $\rel I$ satisfies $\rel I\le \rel A$ \cite{FV98,KZ17book}, and the \emph{promise CSP} with a template $\rel A, \rel B$ (which is a pair of structures such that $\rel A \le \rel B$) asks to distinguish between the cases $\rel I \le \rel A$ and $\rel I \not\le \rel B$ \cite{AGH17,KO22}; the promise being that input falls into one of the two (disjoint) cases.

A \emph{(thin) functor} from the class of structures of some signature to the class of structures of a possibly different signature is a mapping that is monotone with respect to the homomorphism preorder. Many well-known examples of constructions in graph theory, such as the arc graph construction, are functors \cite{FT13,FT15}.
In this paper we investigate a kind of homomorphism duality of such functors called \emph{adjunction} --- two functors $\Lambda$ and $\Gamma$ form an adjoint pair if $\Lambda(\rel H)$ maps homomorphically to $\rel G$ if and only if $\rel H$ maps homomorphically to $\Gamma(\rel G)$ for all relational structures $\rel G, \rel H$ of the corresponding signatures.
In this case, $\Lambda$ is called a left adjoint of $\Gamma$ and $\Gamma$ is a right adjoint of $\Lambda$. If, for some $\Gamma$, there exists such $\Lambda$ we also say that $\Gamma$ has (or admits) a right adjoint.

Adjunction is a core concept of category theory, although we should note that here we work with posetal categories of relational structures, i.e., only existence of a homomorphism matters, rather than with the category of relational structures together with homomorphisms. Adjoints between categories of relational structures in the stricter categorical sense have been completely described by Pultr \cite{Pul70}.
Here, we call these functors \emph{left and central Pultr functors} (following the nomenclature of Foniok and Tardif \cite{FT15}) --- a left Pultr functor is a left adjoint to a central Pultr functor. It is clear, that each pair of left and central Pultr functors is adjoint in our sense as well. Nevertheless, there are more functors that admit a right adjoint than left Pultr functors, e.g., the afore-mentioned arc graph construction.
In this paper, we ask the question which central Pultr functors (i.e., the functors which admit a left adjoint by Pultr's characterisation) admit a right adjoint?

A necessary condition for a central Pultr functor to have a right adjoint was given by Foniok and Tardif \cite{FT15}. They also provided explicit constructions of the right adjoints for some of such functors mapping digraphs to digraphs. Loosely speaking, they gave explicit constructions of right adjoints to functors that either do not change the domain (i.e., the set of vertices) of the digraph, or such that the new domain is the set of edges (arcs) of the input digraph.
In the present paper, we extend their results by proving an explicit construction that works with general relational structures. In particular, we provide adjoints to some of the functors between classes of relational structures satisfying the above-mentioned necessary condition, and that either do not change the domain of the structure, or the new domain is one of the relations of the input structure.
Hence we prove a direct generalisation of \cite{FT15} to arbitrary relational signatures (see Sections~\ref{sec:adjoint-1} and \ref{sec:adjoint-edge}).
Furthermore, by composing right adjoints constructed in this way, we obtain more adjunctions than \cite{FT15} even for the digraph case (see Section~\ref{sec:composition}).
Finally, we believe that our constructions are more intuitive than those provided by \cite{FT15} --- in particular, the above mentioned necessary condition relates adjunction with \emph{finite duality}, and therefore a construction of a right adjoint is related to a construction of a \emph{dual of a tree}. Moreover, every finite tree can be built via natural inductive process. Our constructions, and the proofs that they work, reflect this inductive process for certain trees used in the definition of the functor.
 
\section{Preliminaries}

We recall some basic definitions and notation. We use the symbol $\subsets X$ to denote the power set of a set $X$.

\subsection{Structures and homomorphisms}

A directed graph can be defined as a pair $\rel G = (G, E^{\rel G})$ where $G$ is the set of vertices of $\rel G$ and $E^{\rel G}\subseteq G\times G$ is the set of edges of $\rel G$. This is a special case of a relational structure with a single relation of arity $2$ as defined below.

\begin{definition}
  A \emph{relational signature} $\tau$ is a tuple of relational symbols $R, S, \dots$ where each symbol is assigned a positive integer, called an arity and denoted by $\ar R, \ar S, \dots$.

  A \emph{relational $\tau$-structure} is a tuple $\rel A = (A; R^{\rel A}, S^{\rel A}, \dots)$, where $A$ is a set called the \emph{domain (or universe) of $\rel A$}, and $R^{\rel A} \subseteq A^{\ar R}$, $S^{\rel A} \subseteq A^{\ar S}$, \dots\ are relations on this domain of the corresponding arity.
\end{definition}

The relational symbols $R, S, \dots$ in $\tau$ are also referred to as \emph{$\tau$-symbols}.
When no confusion can arise, we say a $\tau$-structure (dropping ``relational''), or even simply a structure, when $\tau$ is clear from the context. Two structures of the same signature are said to be \emph{similar}.
We will call elements from the domain of some structure \emph{vertices}, and tuples in a relation $R$ of some structure \emph{$R$-edges}, or simply \emph{edges} if the symbol $R$ is either irrelevant or clear from the context.

In the rest of the paper, we will not use the symbol $V$ as a relational symbol --- we restrict its use to refer to the domain of the structure (e.g., in Definition~\ref{def:small-structures} and Section~\ref{sec:terms} below).

Loosely speaking, a homomorphism between two similar structures is a map that preserves relations, e.g., a graph homomorphism would be a map between the two graphs that preserves edges. Formally, a homomorphism is defined as follows.

\begin{definition}
  Let $\rel A$ and $\rel B$ be two structures of the same signature $\tau$. A \emph{homomorphism} $f\colon \rel A \to \rel B$ is defined to be a mapping
  $f\colon A \to B$ such that, for each relational symbol $R$ in $\tau$ and each $(a_1,\dots, a_{\ar R}) \in R^{\rel A}$, we have
  \[
    (f(a_1), \dots, f(a_{\ar R})) \in R^{\rel B}.
  \]
  We will write $\rel A \to \rel B$ if there exists a homomorphism from $\rel A$ to $\rel B$. The set of all homomorphisms from $\rel A$ to $\rel B$ is denoted by $\hom(\rel A, \rel B)$.
\end{definition}

The above definition can be rephrased by saying that the function $f \colon A \to B$ has a coordinate-wise action on each of the relations $R$, i.e., for each relational symbol $R$ in the signature, the expression
\[
    f^R((a_1,\dots, a_{\ar R})) = (f(a_1), \dots, f(a_{\ar R}))
\]
defines a function $f^R\colon R^{\rel A} \to R^{\rel B}$. We use the symbol $f^R$ throughout this paper.

Finally, since we will be extensively working with the homomorphism preorder, this in particular means that we will often work with structures up to homomorphic equivalence --- we say that two structures $\rel A$ and $\rel B$ are \emph{homomorphically equivalent} if we have $\rel A \to \rel B$ and $\rel B \to \rel A$.  Such structures would be identified if we followed the standard procedure to turn the homomorphism preorder into a proper partial order.
Note that two structures $\rel A$ and $\rel B$ are homomorphically equivalent if and only if for every structure $\rel C$, we have $\rel C \to \rel A$ if and only if $\rel C \to \rel B$, i.e., they allow homomorphisms from the same structures. The same is also true for allowing homomorphisms to the same structures, i.e., $\rel A$ and $\rel B$ are homomorphically equivalent if and only if, for all $\rel C$, we have $\rel A \to \rel C$ if and only if $\rel B \to \rel C$. A structure is called a \emph{core}, if it is not homomorphically equivalent to any of its proper substructures.

Certain structures called \emph{trees} play a special role in this paper. We use a definition of a tree equivalent to the one given in \cite[Section 3]{NT00}.
Loosely speaking, a relational structure is a \emph{tree} if it is connected, contains no cycles, and none of its relations has tuples with repeated entries. This is more precisely defined by using the incidence graph of a structure.

\begin{definition}
  The \emph{incidence graph} of a structure $\rel A = (A; R^{\rel A}, \dots)$ is a bipartite multigraph whose vertex set is the disjoint union of $A$ and $R^{\rel A}$ for each relational symbol $R$. There is an edge connecting every tuple $(a_1, \dots, a_k)\in R^{\rel A}$ with every one of its coordinates $a_i \in A$.
  In particular, if some element appears multiple times in this tuple, then the edge connecting it to the tuple appears with the same multiplicity. In this case, the incidence graph contains a cycle of length 2 (i.e., two parallel edges).

  A $\tau$-structure is a ($\tau$-)\emph{tree} if its incidence graph is a tree, i.e., an acyclic connected digraph.
\end{definition}

\begin{remark}
  An undirected graph can be viewed as a relational structure with a single relation $E$ whose universe is the set of vertices $V$. The relation $E \subseteq V \times V$ contains for every edge two tuples $(u,v)$ and $(v,u)$. This means that no undirected graph with at least one edge is a tree according to the above definition since if $(u,v)$ is an edge, then $u$, $(u,v)$, $v$, $(v,u)$ is a 4-cycle of the incidence graph.  Intuitively, relational structures with binary relations are directed graphs; an undirected graph is encoded as directed by including both orientations of each edge which results in a directed cycle of length 2.
  Under the above definition, a directed graph is a tree if it is an oriented tree.
\end{remark}

We now fix notation for some small structures that will be used later.

\begin{definition} \label{def:small-structures}
Fix a relational signature $\sigma$. We define the following structures:
    \begin{itemize}
        \item $\rel V_1$ is the structure with a single vertex, i.e., $V_1 = \{ 1 \}$, and empty relations, i.e., $R^{\rel V_1} = \emptyset$ for all $\sigma$-symbols $R$.
        \item Let $S$ be a relational symbol, $\rel S_1$ is a structure with $\ar S$ vertices related by $S$ and all other relations empty. More precisely, $S_1 = \{ 1, \dots, k \}$ where $k = \ar_S$, $S^{\rel S_1} = \{ (1, \dots, k) \}$, and $R^{\rel S_1} = \emptyset$ for all $\sigma$-symbols $R$ except $S$.
    \end{itemize}
\end{definition}

Since adjunction is a form of duality, we will often mention homomomorphism duality of structures.

\begin{definition}
  A pair $(\rel T, \rel D)$ of similar structures is called a \emph{duality pair} if, for all structures $\rel A$ similar to $\rel T$, either $\rel T \to \rel A$ or $\rel A \to \rel D$. In this case, $\rel D$ is called a dual to $\rel T$.
\end{definition}

It was shown in \cite{NT00} that a structure has a dual if and only if this structure is homomorphically equivalent to a tree. We will also give an alternative proof of one of the implications in Section~\ref{sec:duals}.

\subsection{Pultr functors: gadget replacement and pp-constructions}

Traditionally, in the CSP literature, pp-constructions would be described in the language of logic using so called \emph{primitive positive formulae} (logical formulae that use only $\exists$, $\wedge$, and $=$). We refer to \cite[Definition 19]{BKW17} for details.  In this paper, we define ``pp-constructions'' using a language similar to \cite[Definitions 2.1--2.3]{FT15}.

\begin{definition} \label{def:pultr-template}
  Let $\sigma$ and $\tau$ be two relational signatures. A $(\sigma, \tau)$-\emph{Pultr template} is a tuple of $\sigma$-structures consisting of $\rel P$, and $\rel Q_R$, one for each $\tau$-symbol $R$, together with homomorphisms $\epsilon_{i, R} \colon \rel P \to \rel Q_R$ for each $\tau$-symbol $R$ and all $i \in \{ 1, \dots, \ar R \}$.
\end{definition}

\begin{definition}
  Given a $(\sigma, \tau)$-Pultr template as above, we define two functors $\Lambda$ and $\Gamma$, called \emph{(left and central) Pultr functors}.
  \begin{itemize}
    \item Given a $\tau$-structure $\rel A$, we define a $\sigma$-structure $\Lambda(\rel A)$ in the following way: For each $a\in A$, introduce to $\Lambda(\rel A)$ a copy of $\rel P$ denoted by $\rel P_a$, and for each $\tau$-symbol $R$ and each $(a_1,\dots, a_k)\in R^{\rel A}$, introduce to $\Lambda(\rel A)$ a copy of $\rel Q_R$ with the image of $\rel P$ under $\epsilon_{i,R}$ identified with $\rel P_{a_i}$ for all $i \in \{ 1, \dots, k \}$.
    \item Given a $\sigma$-structure $\rel B$, we define a $\tau$-structure $\Gamma(\rel B)$ whose universe consists of all homomorphisms $h\colon \rel P \to \rel B$. The relation $R^{\Gamma(\rel B)}$ where $R$ is a~$\tau$-symbol is then defined to contain all tuples $(h_1, \dots, h_k)$ of such homomorphisms for which there is a homomorphism $g\colon \rel Q_R \to \rel B$ such that $h_i = g\circ \epsilon_{i,R}$ for all $i \in \{1, \dots, k\}$.
  \end{itemize}
\end{definition}

Once the definitions are settled, it is not too hard to show that, for any Pultr template, the corresponding Pultr functors $\Lambda$ and $\Gamma$ are left and right adjoints. This statement is attributed to \cite{Pul70}, although it was rediscovered on numerous occasions, and it is considered folklore in category theory.

The fact that both $\Lambda$ and $\Gamma$ are (thin) functors follows from the adjunction. It can be also easily proved directly, for example, a homomorphism $f^\Gamma\colon \Gamma(\rel A) \to \Gamma(\rel B)$ can be obtained from a homomorphism $f\colon \rel A \to \rel B$ by setting $f^\Gamma(h) = f\circ h$ for each $h\colon \rel P \to \rel A$.

\begin{example} \label{ex:arc-graph}
  Consider the Pultr template consisting of structures $\rel P = (\{0,1\}; \{(0,1)\})$ and $\rel Q_E = (\{0,1,2\}; \{(0,1), (1,2)\})$ with $\epsilon_1(0) = 0$ and $\epsilon_1(1) = 1$, and $\epsilon_2(0) = 1$ and $\epsilon_2(1) = 2$.

  The corresponding central Pultr functor $\Gamma$ is the arc-graph construction that is usually denoted by $\delta$; given a (directed) graph $\rel G = (G, E^{\rel G})$, the digraph $\delta(\rel G)$ is defined as the graph with the vertex set $E^{\rel G}$ and edges $((u,v), (v,w)) \in E^{\rel G} \times E^{\rel G}$.

  The left Pultr functor $\Lambda$ corresponding to this template provides a left adjoint to the arc-graph construction, and can be explicitly described as follows: Given a digraph $\rel G$, the digraph $\Lambda(\rel G)$ is constructed by replacing each vertex $v \in G$ with a pair of vertices $v_0, v_1$ connected by an edge, i.e., with $(v_0, v_1) \in E^{\Lambda(\rel G)}$. Furthermore, for each edge $(u, v) \in E^{\rel G}$, we identify the vertices $u_1$ and $v_0$.
  For example, the image of a path of length $n$, i.e., the digraph $\rel P_n = (\{0, \dots, n\}; \{(i, i+1) \mid i \in \{0, n-1\})$ is the path of length $n + 1$ consisting of vertices $0_0, 0_1 = 1_0, 1_1 = 2_0, \dots, n_1$.

  Observe that, in this case, we have that $\Lambda(\Gamma(\rel P_n))$ is isomorphic to $\rel P_n$. This does not need to happen for any digraph in place of $\rel P_n$, nevertheless, the existence of a homomorphism $\Lambda(\Gamma(\rel G)) \to \rel G$ for all digraphs $\rel G$ follows from the adjunction in a straightforward way.
\end{example}

We remark that the connection between the adjunction of left and central Pultr functors and the algebraic reductions between CSPs has been described in \cite[Section 4.1]{KOWZ20}. Let us briefly mention that the central Pultr functors are \emph{pp-constructions} (more precisely, a structure $\rel A$ is pp-constructible from $\rel B$ if it is homomorphically equivalent to the structure $\Gamma(\rel B)$ for some Pultr functor $\Gamma$), and the left Pultr functors are called \emph{gadget replacements} in this context.

The main contribution of the present paper is an investigation of the cases of Pultr templates, for which $\Gamma$ is also a left adjoint, i.e., it admits a right adjoint $\Omega$.
The following necessary condition for this was proved in \cite{FT15} for the case of digraphs, but the proof goes through verbatim for general structures; we include it for completeness.

\begin{theorem}[{\cite[Theorem 2.5]{FT15}}] \label{thm:adjoints-and-duals}
  Let $\Lambda$ and $\Gamma$ be a pair of left and central Pultr functors defined by a $(\sigma, \tau)$-Pultr template. If $\Gamma$ has a right adjoint, then, for each $\tau$-tree $\rel T$, $\Lambda(\rel T)$ is homomorphically equivalent to a $\sigma$-tree.
\end{theorem}

\begin{proof}
  A structure has a dual if and only if it is homomorphically equivalent to a tree \cite{NT00}.
  Thus, it suffices to prove that if $\Gamma$ has both a left adjoint $\Lambda$ and a right adjoint $\Omega$, and $(\rel T, \rel D)$ is a duality pair, then $(\Lambda(\rel T), \Omega(\rel D))$ is a duality pair as well. To show that, observe that, for any $\rel A$,  $\Lambda(\rel T) \not\to \rel A$ is equivalent to $\rel T \not\to \Gamma(\rel A)$ since $\Lambda$ is a left adjoint to $\Gamma$. The latter condition is equivalent to $\Gamma(\rel A) \to \rel D$, since $\rel D$ is a dual of $\rel T$, which in turn is equivalent to $\rel A \to \Omega(\rel D)$, since $\Omega$ is a right adjoint to $\Gamma$.
\end{proof}

\begin{corollary}
  If a central Pultr functor $\Gamma$ has a right adjoint, then all the structures in its Pultr template are homomorphically equivalent to trees.
\end{corollary}

\begin{proof}
  As shown in \cite{FT15} (and not hard to see), if the template consists of structures $\rel P$ and $\rel Q_R$ for $\tau$-symbols $R$, then $\rel P$ is isomorphic to $\Lambda(\rel V_1)$ and $\rel Q_R$ is isomorphic to $\Lambda(\rel R_1)$ for each $\tau$-symbol $R$ (where tree $\tau$-structures $\rel V_1$ and $\rel R_1$ are as in Definition~\ref{def:small-structures}).
\end{proof}

It is open whether the condition in Theorem~\ref{thm:adjoints-and-duals} is also sufficient to have a right adjoint. In this paper, as well as Foniok-Tardif do in \cite{FT15}, we focus on the cases when $\rel P$ and $\rel Q_R$'s are actually trees, and we give a concrete construction of the adjoint in two cases: $\rel P = \rel V_1$ (this is shown in Section~\ref{sec:adjoint-1}), and $\rel P = \rel S_1$ is the tree with a single $S$-edge for some $\sigma$-symbol $S$  (this is shown in Section~\ref{sec:adjoint-edge}). Finally, in the last section, we combine these two constructions to prove that if $\rel P$ and $\rel Q_R$'s are trees, and, moreover, the images of $\rel P$ in $\rel Q_R$ under the maps $\epsilon_{i,R}$ are either disjoint or intersect in one vertex, then the corresponding central Pultr functor has a right adjoint (Theorem~\ref{thm:7.1}).

Finally, let us note a well-known fact from category theory that if  $\Omega_1$ and $\Omega_2$ are both right adjoints to $\Gamma$, then $\Omega_1(\rel A)$ and $\Omega_2(\rel A)$ are homomorphically equivalent for all structures $\rel A$. The proof is immediate from the definitions: both structures allow a homomorphism from $\rel B$ if and only if $\Gamma(\rel B) \to \rel A$.
 
\section{An inductive construction of trees}
    \label{sec:terms}

For our constructions, it will be convenient to describe $\tau$-trees by certain formal terms. Our terms will correspond to trees that are rooted, either in a vertex or in an edge (i.e., a tuple in one of the relations). Each term will correspond to a unique inductive construction of a tree, but the same tree can be obtained by several inductive constructions.

We recall that we assume that none of the relational signatures uses $V$ as a relational symbol. Each term will be either a $V$-term ($V$ stands for \emph{vertices}) or an $R$-term where $R$ is a relational symbol in $\tau$. The rules for the inductive construction of terms and their semantics are as follows:

\begin{itemize}
    \item $\vertex$ is a $V$-term.
    \item If $t_1,\ldots,t_k$ are $V$-terms and $R$ is a $k$-ary symbol in $\tau$, then $\edge_R(t_1,\ldots,t_k)$ is an $R$-term.
    \item If $t$ is an $R$-term and $i\in \{ 1, \dots, \ar R\}$, then $\pr_i(t)$ is a $V$-term.
\end{itemize}

The rooted tree $(\relT t,r_t)$, where $r_t$ is the root, corresponding to a term $t$ is defined in the following way:
\begin{itemize}
    \item $\relT \vertex$ is the one-vertex tree, rooted at its only vertex, i.e., $\relT \vertex = \rel V_1$ and $r_t = 1$.\item If $t = \edge_R(t_1,\ldots,t_k)$ is an $R$-term, then $\relT t$ is the tree obtained by taking the disjoint union of the trees $\relT{t_i}$ (with roots $r_{t_i}$) and adding the tuple $(r_{t_1}, \dots, r_{t_k})$ to the relation $R$. The root $r_t$ is defined to be this new tuple.
    \item If $t = \pr_i(s)$ is a $V$-term and $r_s = (v_1, \dots, v_k)$, then we set $\relT t = \relT s$ and $r_t = v_i$.
\end{itemize}
The domain of $\relT t$ is denoted by $T(t)$.

We say that $t$ \emph{represents a tree $\rel T$} if $\relT t$ is isomorphic to $\rel T$, and that $t$ \emph{represents a tree $\rel T$ rooted in $r$} if $\relT t$ is isomorphic to $\rel T$ via an isomorphism mapping $r_t$ to $r$. The same tree can in general be represented by several terms (see the following example), but each term represents a unique tree up to isomorphism.

\begin{example} \label{ex:terms-and-trees}
\begin{figure}
\begin{gather*}
  \begin{matrix}
    \begin{tikzpicture}[every node/.style = {circle, inner sep=1.3, outer sep=1},
        baseline = {([yshift=-.8ex]current bounding box.center)}]
      \node [draw] (b1) at (1,0) {};
      \node [draw] (b2) at (2,0) {};
      \node [draw] (b3) at (3,0) {};
      \draw [->] (b1) -- (b2);
      \draw [->, very thick, red] (b2) -- (b3);
    \end{tikzpicture}
    &
    \begin{tikzpicture}[every node/.style = {circle, inner sep=1.3, outer sep=1},
        baseline = {([yshift=-.8ex]current bounding box.center)}]
      \node [draw] (b2) at (1,0) {};
      \node [draw] (b1) at (2,0) {};
      \node [draw] (b3) at (3,0) {};
      \draw [->] (b2) -- (b1);
      \draw [->, very thick, red] (b3) -- (b1);
    \end{tikzpicture}
    \\*
    \edge_E(\pr_2 \edge_E(\vertex, \vertex), \vertex) &
    \edge_E(\vertex, \pr_2 \edge_E(\vertex, \vertex))
  \end{matrix} \\
\begin{tikzpicture}[every node/.style = {circle, inner sep=1.3, outer sep=1},
        baseline = {([yshift=-.8ex]current bounding box.center)}]
      \node [draw] (1) at (1,0) {};
      \node [draw] (2) at (0,0) {};
      \node [draw] (3) at (2,0) {};
      \node [draw] (4) at (-.5,-.866) {};
      \node [draw] (5) at (-.5,.866) {};
      \draw [->, very thick, red] (2) -- (1);
      \draw [->] (3) -- (1);
      \draw [->] (4) -- (2);
      \draw [->] (2) -- (5);
    \end{tikzpicture} \\*
    \edge_E\bigl(
        \pr_1\edge_E(\pr_2\edge_E(\vertex, \vertex), \vertex),
        \pr_2\edge_E(\vertex, \vertex)\bigr)
\end{gather*}
  \caption{Examples of terms and trees they represent.}
  \label{fig:trees}
\end{figure}
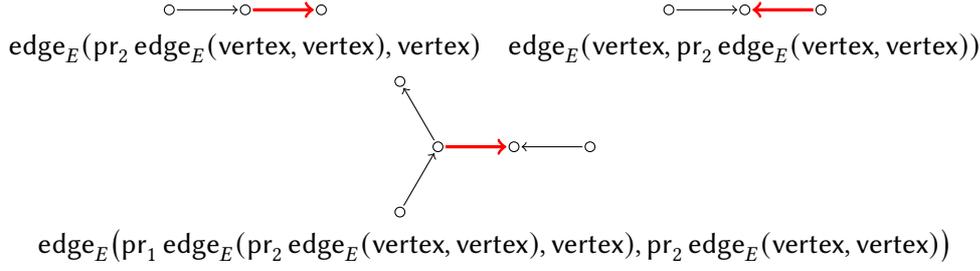

The following are examples of terms and the trees they represent in the relational signature of digraphs. The corresponding roots are highlighted. We drop the brackets around arguments of $\pr_i$ to ease readability.
  
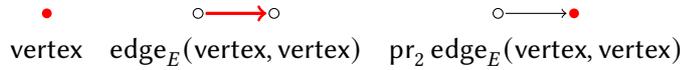
\[
  \begin{matrix}
    \begin{tikzpicture}[every node/.style = {circle, inner sep=1.3, outer sep=1},
        baseline = {([yshift=-.8ex]current bounding box.center)}]
      \node [fill=red] (b2) at (0,0) {};
    \end{tikzpicture}
    &
    \begin{tikzpicture}[every node/.style = {circle, inner sep=1.3, outer sep=1},
        baseline = {([yshift=-.8ex]current bounding box.center)}]
      \node [draw] (b2) at (0,0) {};
      \node [draw] (b4) at (-1,0) {};
      \draw [->, red, very thick] (b4) -- (b2);
    \end{tikzpicture}
    &
    \begin{tikzpicture}[every node/.style = {circle, inner sep=1.3, outer sep=1},
        baseline = {([yshift=-.8ex]current bounding box.center)}]
      \node [fill=red] (b2) at (0,0) {};
      \node [draw] (b4) at (-1,0) {};
      \draw [->] (b4) -- (b2);
    \end{tikzpicture}
    \\
    \vertex &
    \edge_E(\vertex, \vertex) &
    \pr_2 \edge_E(\vertex, \vertex)
  \end{matrix} \\
\]
Note that the difference between the second tree and the third tree is the root; the latter tree is obtained from the other by applying $\pr_2$, and thus only changing the root. Further, we present a few examples of more complicated trees in Fig.~\ref{fig:trees}.

We remark that a tree can be represented by multiple terms even when we fix the root. For example, the last tree in Fig.~\ref{fig:trees} can be also represented by the term
\[
    \edge_E\bigl(
        \pr_2\edge_E(\vertex, \pr_1\edge_E(\vertex, \vertex)),
        \pr_2\edge_E(\vertex, \vertex)\bigr).
\]
\end{example}

\begin{lemma}
  Fix a relational signature $\tau$. Any finite $\tau$-tree can be represented by a term. More precisely,
  \begin{itemize}
    \item for every finite tree $\rel T$ and $r\in T$, there is a $V$-term $t$ such that $\rel T(t)$ is isomorphic to $\rel T$ via an isomorphism that maps $r_t$ to~$r$, and
    \item for every finite tree $\rel T$ and $r\in R^{\rel T}$ for some $\tau$-symbol $R$, there is an $R$-term $t$ such that $\rel T(t)$ is isomorphic to $\rel T$ via an isomorphism that maps $r_t$ to~$r$.
  \end{itemize}
\end{lemma}

The proof is a simple argument by induction on the number of edges of the tree. We include it in detail to provide more intuition about terms.

\begin{proof}
  We prove this statement by induction on the number of edges of $\rel T$. Each induction step is moreover split in two: we first prove that trees with $n$ edges rooted in an edge can be represented, and then, assuming the above, we show that trees with $n$ edges rooted in a vertex can be represented.

  \begin{enumerate}
    \item We start with $n = 0$. There is a single tree $\rel T$ with no edges, namely the tree with one vertex. It cannot be rooted in an edge, but it can be rooted in its only vertex. It is represented by the term $\vertex$.
    \item Assume $n > 0$, and we can represent all trees rooted in a vertex with less than $n$ edges. Assume that $\rel T$ has $n$ edges and is rooted in an edge $(r_1,\dots, r_k) \in R^{\rel T}$ for some $k$-ary symbol $R$. Removing this edge from $\rel T$ splits $\rel T$ into $k$ connected components $\rel T_1, \dots, \rel T_k$, where $\rel T_i$ contains $r_i$ for each $i$. Each of these components is a tree. Assuming that $\rel T_i$ rooted in $r_i$ is represented by the term $t_i$ for each $i$, $\rel T$ rooted in $(r_1,\dots, r_k)$ is represented by $\edge_R(t_1,\dots, t_k)$.
    \item Assume that $n > 0$, and we can represent all trees rooted in an edge with at most $n$ edges. Let $\rel T$ be a tree with $n$ edges and $r \in T$. Since $\rel T$ is connected, $r$ is involved in some edge, say $(v_1,\dots, v_k)\in R^{\rel T}$ where $k = \ar R$ and $r = v_i$ for some $i$. Now, $\rel T$ rooted in $(v_1, \dots, v_k)$ is represented by an $R$-term $t$ by the inductive assumption, so $\rel T$ rooted in $r$ is represented by $\pr_i(t)$. \qedhere
  \end{enumerate}
\end{proof}

Finally, our constructions use the notion of a \emph{subterm} of a term. Intuitively a subterm of a term $t$ is any proper term that appears as a part of $t$. The set of all subterms of $t$ encodes all intermediate byproducts of the inductive construction of $\relT t$. Formally, we define subterms as follows.
\begin{itemize}
    \item The only subterm of $\vertex$ is itself.
    \item If $t = \edge_R(t_1,\ldots,t_k)$ is an $R$-term, where $t_1,\ldots,t_k$ are $V$-terms and $R$ is a $k$-ary symbol in $\tau$, then the set of its subterms consists of the term itself and all subterms of $t_1,\ldots,t_k$.
    \item If $t = \pr_i(s)$ is a $V$-term, where $s$ is an $R$-term, then the set of its subterms consists of the term itself and all subterms of $s$.
\end{itemize}

If $s$ is a subterm of $t$, then we write $s\leq t$, moreover, if $s$ is a proper subterm of $t$, i.e., it is a subterm and $s\neq t$, we write $s < t$. A $V$-subterm of a term $t$ is a subterm which is a $V$-term, and similarly an $R$-subterm for a relational symbol $R$ is a subterm which is an $R$-term.

For example, the term
\[
    t = \edge_E\bigl(\pr_1(\edge_E(\vertex, \vertex)), \pr_1(\edge_E(\vertex, \vertex))\bigr)
\]
has four distinct subterms: two $E$-terms, which are $t$ and $\edge_E(\vertex, \vertex)$, and two $V$-terms, which are $\pr_1(\edge_E(\vertex, \vertex))$ and $\vertex$.

We note that statements about terms can be proven by an inductive principle: showing the statement is true for the term $\vertex$, and then showing that if it is true for all proper subterms of a term $t$, it is also true for $t$.
 
\section{Prelude: Duals to trees}
  \label{sec:duals}

Before we get to the main construction of adjoints, let us briefly discuss a simpler construction of a dual of a tree. There are a few similarities between the construction of duals and right adjoints to central Pultr functors: as we mentioned before (see the proof of Theorem~\ref{thm:adjoints-and-duals}), if $\Gamma$ is a central Pultr functor that has a left adjoint $\Lambda$ and a right adjoint $\Omega$, and $(\rel T, \rel D)$ is a duality pair, then $(\Lambda(\rel T), \Omega(\rel D))$ is also a duality pair. Moreover, our construction of the dual uses the inductive construction of trees from the previous section in a similar way as the constructions of right adjoints in Sections~\ref{sec:adjoint-1} and \ref{sec:adjoint-edge}, but the construction of a dual is conceptually easier, so we present it to create some intuition that will be useful below. We also note that we show an explicit connection between constructions of a dual and the construction of a right adjoint in Section~\ref{sec:adjoints-and-duals}.

Again, we fix a relational signature.
Recall that a pair $(\rel T, \rel D)$ of similar structures is called a duality pair if, for all structures $\rel A$ similar to $\rel T$, either $\rel T \to \rel A$ or $\rel A \to \rel D$. In this case, $\rel D$ is called a dual to $\rel T$, and it was shown in \cite{NT00} that a structure has a dual if and only if it is homomorphically equivalent to a tree.
See \cite{LLT07,NT05} for other characterisations of finite duality.
Our construction is loosely inspired by the construction in \cite{NT05}.

Another way to look at duals is that for any structure $\rel A$, a homomorphism $\rel A \to \rel D$ should correspond to a~`\emph{proof}' that $\rel T \not\to \rel A$. How can one prove that a tree does not map to a structure $\rel A$ if that is the case? This can be done by an inductive argument. Let us outline this argument in the case $\rel T$ and $\rel A$ are digraphs. More precisely, we describe a procedure that shows that, for a fixed root $r \in T$ and some $a\in A$, there is no homomorphism from $\rel T$ to $\rel A$ that maps $r$ to $a$. Pick a neighbour $s$ of $r$ in $\rel T$, and assume that $(r,s) \in E^{\rel T}$, the other orientation is dealt with symmetrically. We can show that there is no homomorphism from $\rel T$ to $\rel A$ mapping $r$ to $a$ by showing that, for none of the neighbours $b$ of $a$ in $\rel A$, there is a homomorphism $\rel T \to \rel A$ that maps the edge $(r,s)$ to $(a,b)$. In turn, removing the edge $(r,s)$ from $\rel T$ splits the tree into two subtrees $\rel T_1$, containing $r$, and $\rel T_2$, containing $s$. A homomorphism $\rel T \to \rel A$ that maps $(r,s)$ to $(a,b)$ is equivalent to a pair of homomorphisms $\rel T_1 \to \rel A$, that maps $r$ to $a$, and $\rel T_2 \to \rel A$, that maps $s$ to $b$. We have thus reduced the claim to proving that there is no homomorphism from (at least) one of the two smaller trees, and can therefore recursively repeat our strategy for these two smaller trees. We further design a structure $\rel D$, in which we can encode proofs of the above form.
In particular, the image of an element $a \in \rel A$ under a homomorphism $\rel A \to \rel D$ will contain answers to questions of the form `Is there a homomorphism $\rel T' \to \rel A$ that maps $r$ to $a$?' for all trees $\rel T'$ and all roots $r\in T'$ that would appear in the above inductive argument.

\begin{definition} \label{def:duals}
  Let $t_Q$ be an $S$-term for some relational symbol $S$.  Let $\trees_V$ be the set of all $V$-subterms of $t_Q$, and let $\trees_R$ be the set of all $R$-subterms of $t_Q$ for each relational symbol $R$.
  We define a structure $\rel D(t_Q)$.

  The domain of this structure, denoted by $D(t_Q)$, is the set of all tuples $v\in \{ \true, \false \}^{\trees_V}$, indexed by $V$-subterms of $t_Q$, such that $v_{\vertex} = \true$.

  To define edges, we introduce the following notation. For terms $t_1, \dots, t_k \in \trees_V$, a relational symbol $R$ of arity $k$, and $v^1, \dots, v^k \in D(t_Q)$, we let
  \[
    \w_{\edge_R(t_1, \dots, t_k)}(v^1,\dots, v^k) = v^1_{t_1} \wedge \dots \wedge v^k_{t_k}.
  \]
  A tuple $(v^1,\dots, v^k)$ of vertices is related in a relation $R$ (of arity $k$) in $\rel D(t_Q)$ if
  \begin{enumerate}
    \item[\normalfont (D1)]  For all $t\in \trees_S$ and $i\in \{1,\dots, k\}$ such that $\pr_i(t) \in \trees_V$, we have
      \[
        \w_t(v^1,\dots, v^k) \impl v^i_{\pr_i(t)}.
      \]
    \item[\normalfont (D2)] If $R = S$, i.e., if $t_Q$ is an $R$-term, then
      \[
        \w_{t_Q}(v^1,\dots, v^k) = \false.
      \]
  \end{enumerate}
\end{definition}

If $e = (v_1, \dots, v_k)$, we will often write $\w_t(e)$ instead $\w_t(v_1, \dots, v_k)$.
Note that if $e \in R^{\rel D(t_Q)}$, then we have a tuple $e^* \in \{ \true, \false \}^{\trees_R}$ defined by $e^*_t = \w_t(e)$, which satisfies $e^*_{t_Q} = \false$ (assuming $t_Q$ is an $R$-term).
This draws a parallel to how vertices are defined.
Finally, note that item (D1) is essentially quantified by the $V$-subterms of $t_Q$, since all of such subterms $t'$, with the exception $t' = \vertex$, are of the form $\pr_i(t)$ for some $t$, and $t$ and $i$ are uniquely defined by $t'$.

\begin{theorem} \label{thm:dual}
  For any relational structure $\rel Q$ that is homomorphically equivalent to $\rel T(t_Q)$ for some $S$-term $t_Q$ where $S$ is a relational symbol, $\rel D(t_Q)$ is a dual of $\rel Q$.
\end{theorem}

\begin{proof}
  Without loss of generality, we may assume that $\rel Q = \rel T(t_Q)$.  We need to show that, for all structures $\rel A$, $\rel A\to \rel D(t_Q)$ if and only if $\rel Q\not\to \rel A$.

  First, assume that $\rel Q \not\to \rel A$. We define $f\colon A \to D(t_Q)$ by putting, for each $u\in A$ and each $t\in \trees_V$, $f(u)_t = \true$ if there is a homomorphism $h\colon \rel T(t) \to \rel A$ that maps the root to $u$, and $f(u)_t = \false$ otherwise. Clearly, $f(u)_{\vertex} = \true$.
  Now, assume that $e = (u_1, \dots, u_k) \in R^{\rel A}$.
  Observe that $\w_t(f^R(e))$ is true if and only if there is a homomorphism $h\colon \rel T(t) \to \rel A$ mapping the root edge to $e$.  In other words, there is a homomorphism $h\colon \rel T(t) \to \rel A$ mapping the root edge to $e$ if and only if there are homomorphisms $h_i \colon \rel T(t_i) \to \rel A$ with $h_i(r_{t_i}) = u_i$ for all $i$. Assuming that such a homomorphism $h$ exists, the homomorphisms $h_i$ are defined as restrictions of $h$ to the corresponding subtrees. Assuming homomorphisms $h_i$ exist, taking the union of $h_i$ defines a mapping $h$ on all vertices of $\rel T(t)$. This mapping clearly preserves all edges different from the root, since $h_i$ are homomorphisms, and it also preserves the root edge, since it is mapped to $e\in R^{\rel A}$. This means that $h$ is indeed a homomorphism. We need to show that $f^R(e)$ satisfies (D1) and (D2).  For (D1), we want
  \[
    \w_t(f^R(e)) \impl f(u_i)_{\pr_i(t)},
  \]
  i.e., if there is a homomorphism $\rel T(t)\to \rel A$ mapping the root to $e$ then there is a homomorphism $\rel T(\pr_i(t)) \to \rel A$ mapping the root to $u_i$. This is trivial, since $\rel T(t) = \rel T(\pr_i(t))$ and a homomorphism $h\colon \rel T(t)\to \rel A$ that maps the root $r_t$ to $(u_1,\dots, u_k)$ necessarily maps $r_{\pr_i(t)}$ to $u_i$.
  Finally, (D2) is clear from the definition, since we assumed that $\rel Q \not\to \rel A$.

  For the other implication, we first prove the following by induction on the term $t \leq t_Q$:
  \begin{claim} \label{claim:duals}
    Let $t \leq t_Q$ and let $h\colon \rel T(t) \to \rel D(t_Q)$ be a homomorphism. Then $h(r_t)_t = \true$ if $t$ is a $V$-term, and $\w_t(h^R(r_t)) = \true$ if $t$ is an $R$-term.
  \end{claim}

  \begin{enumerate}
    \item The case $t = \vertex$ is trivial.
    \item Let $t = \edge_R (t_1, \dots, t_{\ar R})$. We assume that $h\colon \rel T(t) \to \rel A$ is a homomorphism and $h(r_t) = (v^1, \dots, v^{\ar R})$. This in particular means that, for each $i$, $h$ maps the root of $\rel T(t_i)$ to $v^i$. Hence, we can apply the inductive assumption on the restrictions of $h$ to $\rel T(t_i)$'s to get that $v^i_{t_i} = \true$ for all $i$, and consequently the claim by the definition of $\w_t(h^R(r_t))$.
    \item Let $t = \pr_{i,R} (t')$, and $h\colon \rel T(t) \to \rel A$. Note that $\rel T(t) = \rel T(t')$, so $h$ is also a homomorphism from $\rel T(t')$. Let $h^R(r_{t'}) = (v^1, \dots, v^k)$, and observe that $h(r_t) = v^i$. By the inductive assumption, this implies that $\w_{t'}(h^R(r_{t'})) = \w_{t'}(v^1, \dots, v^k) = \true$. Consequently, we get that $h(r_t)_t = v^i_t = \true$ from (D1).
  \end{enumerate}

  This concludes the proof of the claim. Assume for a contradiction that $f\colon \rel A \to \rel D(t_Q)$ and $g\colon \rel Q \to \rel A$ are homomorphisms. Since $\rel T(t_Q) = \rel Q$, the above claim applied to $t = t_Q$ and the homomorphism $h = f\circ g\colon \rel Q \to \rel D(t_Q)$, would imply that $\w_{t_Q}(h^R(r_{t_Q})) = \true$ which would contradict (D2).
\end{proof}

\subsection{Example: Dual to a directed path}
  \label{ex:path-dual}

It is well known (see, e.g., \cite[Proposition 1.20]{HN04}) that a directed graph maps homomorphically to the graph $\rel L_k = (\{1, \dots, k\}; <)$, where the edge relation is given by the strict order on the domain, if and only if it does not allow a homomorphism from a directed path with $k$ edges (and $k+1$ vertices) --- we denote this path by $\rel P_k$. This means that $\rel L_k$ is the dual of $\rel P_k$. Let us compare this observation with our construction of the dual.

Fix $k > 0$. We pick the term
\begin{gather*}
  t_k = \edge_E\left(\pr_2\bigl(\edge_E\bigl(\dots
      \pr_2(\edge_E(\vertex, \vertex)), \dots, \vertex\bigr)\bigr), \vertex\right),
\end{gather*}
where $\edge_E$ appears $k$ times, to represent $\rel P_k$ rooted in its last edge, i.e., the following graph.
\[
  \begin{tikzpicture}[every node/.style = {circle, inner sep=1.3, outer sep=1},
      baseline = {([yshift=-.8ex]current bounding box.center)}]
    \node [draw] (b1) at (1,0) {};
    \node [draw] (b2) at (2,0) {};
    \node [draw] (b3) at (3,0) {};
    \node [draw] (b4) at (4,0) {};
    \node [draw] (b5) at (5,0) {};
    \draw [->] (b1) -- (b2);
    \node at (2.5,0) {$\dots$};
    \draw [->] (b3) -- (b4);
    \draw [->, very thick, red] (b4) -- (b5);
  \end{tikzpicture}
\]
We also let $s_0 = \vertex$, $t_i = \edge_E(s_{i-1}, \vertex)$ and $s_i = \pr_2 (t_i)$ for $i \in \{1, \dots, k-1\}$.
Note that both $t_i$ and $s_i$ represent a path of length $i$ --- the difference is the root which is either the last edge or the last vertex.
Finally, note that $s_0$, \dots, $s_{k-1}$ and $t_1$, \dots, $t_k$ are the only subterms of $t_k$.

By definition, a vertex of $\rel D(t_k)$ is a tuple
\[
  u \in \{\true, \false \}^{\{s_0, \dots, s_{k-1}\}}
\]
such that $u_{s_0} = \true$. This allows us to write them simply as ordered $k$-tuples whose $i$-th entry is the value corresponding to $s_{i-1}$.

To check whether $u$ is connected by an edge to $v$ or not, we consider the expressions $\w_{t_i}(u, v) = u_{s_{i-1}} \wedge v_{s_0}$ for all $i$.
Using the definition, we get that $(u,v)$ is an edge if
\begin{enumerate}
  \item[(D1)] for all $i$, $\w_{t_i}(u,v) \impl v_{s_i}$ since $s_i = \pr_2(t_i)$,\footnote{The implication $\w_{t_i}(u,v) \impl v_{\pr_1(t_i)}$ is not needed since $\pr_1(t_i)$ is not a subterm of $t_k$.}
    and
  \item[(D2)] $\w_{t_k}(u,v) = \false$.
\end{enumerate} 
Since $v_{s_0} = \true$, we may simplify $\w_{t_i} (u, v) = u_{s_{i-1}}$, and substitute into the conditions above:
\begin{enumerate}
  \item $u_{s_{i-1}} \impl v_{s_i}$, for all $i<k$, and
  \item $u_{s_{k-1}} = \false$.
\end{enumerate}
In particular, observe that $u$ has no out-edge (i.e., there is no edge of the form $(u,v)$ for any $v$) if its last entry is $\true$.

It is not hard to see that $\rel L_k$ maps to a dual constructed this way. We can construct a homomorphism $h$ by mapping $i \in \{1, \dots, k\}$ to the tuple starting with $i$ $\true$'s and followed by all $\false$'s, i.e.,
\begin{align*}
  h(1) &= (\true, \false, \dots, \false) \\
  h(2) &= (\true, \true, \false, \dots, \false) \\
  \vdots\\
  h(k) &= (\true, \true, \dots, \true)
\end{align*}
Note that this is exactly the same homomorphism that is constructed in the proof of Theorem \ref{thm:dual} (assuming $\rel A = \rel L_k$). And indeed, it is easy to check that if $i < j$, then $h(i)$ and $h(j)$ satisfy the conditions for an edge given above.

Naturally, there is also a homomorphism the other way. One such homomorphism maps a tuple $u$ that begins with $i$ $\true$'s followed by a $\false$ to $i$. Again, it is easy to check that if $(u,v)$ is an edge then $v$ has to begin with one more $\true$. This establishes that our construction is homomorphically equivalent to $\rel L_k$ (as it should be).

\begin{remark}
Let us note that any homomorphism constructed according to the proof of Theorem~\ref{thm:dual} uses only the vertices in the image of $h$ above, i.e., the tuples of the form
  \[
    (\true, \dots, \true, \false, \dots,\false),
  \]
  where $\true$ appears at least once and $\false$ does not need to appear at all. This is quite easy to see: $\rel T(t_i)$ maps to $\rel T(t_j)$ for all $i < j$ via a homomorphism preserving the roots, hence, for any $u$ in the image, we get that if $u_{t_j} = \true$ then $u_{t_i} = \true$ for all $i < j$.

  We could force similar implications in the definition of the dual by requiring that $u_t \impl u_s$ whenever there is a homomorphism $\rel T(s) \to \rel T(t)$ preserving roots. We did not include this condition in the definition because our goal is to get a simple construction and not necessarily that the construction results in the smallest graph possible. Nevertheless, this raises a question: Would it be possible by enforcing such implications on our general construction to produce a dual that would be a \emph{core} (i.e., a structure that is not homomorphically equivalent to any of its proper substructures)?
\end{remark}

Finally, we note that our construction of the dual of a tree can be naturally extended to any (finite) tree duality, i.e., given a finite set of finite trees $\mathscr F = \{\rel T_1, \dots, \rel T_n\}$, we can construct their dual $\rel D$ that will satisfy that, for any structure $\rel A$, $\rel A \to \rel D$ if and only if for all $i = 1, \dots, n$, $\rel T_i \not\to \rel A$.
This $\rel D$ is constructed as in Definition~\ref{def:duals} with the following changes: Assume that a term $t_i$ represents $\rel T_i$ rooted in an edge for each $i$ (we are assuming that none of $\rel T_i$'s consists of a single vertex). Let $\mathscr T$ be the set of all subterms of any of the $t_i$'s. We use notation $\trees_V$ and $\trees_R$ (where $R$ is a relational symbol) for the $V$-terms and $S$-terms, respectively, that belong to this set.
Finally, replace the condition (D2) with
\begin{enumerate}
  \item[\normalfont (D2')] $\w_{t_i}(e) = \false$ for all $i = 1, \dots, n$ (naturally, this only applies if $t_i$ is an $R$-term).
\end{enumerate}
We denote a structure constructed this way by $\rel D(t_1, \dots, t_n)$.

\begin{theorem}
  Let $q_1, \dots, q_n$ be terms for each $i \in \{1, \dots, n\}$, and assume that $q_i$ is an $R_i$-term, where $R_i$ is a relational symbol. For each structure $\rel A$, we have that either there exists $i$ such that $\relT {q_i} \to \rel A$ or $\rel A \to \rel D(q_1, \dots, q_n)$.
\end{theorem}

\begin{proof}
  The proof is essentially identical to the proof of Theorem~\ref{thm:dual}.
  First, the implication `$\relT{q_i} \not\to \rel A$ for all $i \in \{1, \dots, n\}$ implies $\rel A \to \rel D(q_1, \dots, q_n)$' is proven by the same argument as before. In particular, we define $f\colon \rel A \to \rel D(q_1, \dots, q_n)$ by $f(u)_t = \true$ if $\rel T(t) \to \rel A$ for each $t\in \trees_V$ and $u\in A$, and $f(u)_t = \false$ otherwise.
  The rest of the proof is identical except that instead of arguing that $\w_t(f^R(e))$ satisfies (D2) using $\rel Q \not\to\rel A$, we argue that $\w_t(f^R(e))$ satisfies (D2') for some $i$ using that $\relT{q_i} \not\to \rel A$.

  For the other implication, it is enough to argue that $\relT{q_i} \not\to \rel D(q_1, \dots, q_n)$ for each $i$, since then $\relT{q_i}$ does not map homomorphically to any $\rel A$ that maps to $\rel D(q_1, \dots, q_n)$.
  Observe that Claim~\ref{claim:duals} and its proof is valid in this case.  Consequently, if $h\colon \relT{q_i} \to \rel D(q_1, \dots, q_n)$ for some $i$, then $\w_{q_i}(h^{R_i}(r_{q_i})) = \true$ by the claim, which would contradict (D2').
\end{proof}
 
\section{Adjoints to functors not changing the domain}
  \label{sec:adjoint-1}

In this section, we describe the simpler of the cases of our construction of an adjoint.
We consider $(\sigma, \tau)$-Pultr templates where $\rel P = \rel V_1$ (i.e., a vertex) and, for each $\tau$-symbol $R$, $\rel Q_R$ is a $\sigma$-tree. The homomorphisms $\epsilon_{i,R}\colon \rel P \to \rel Q_R$ are given by picking elements $x_1, \dots, x_{\ar R} \in Q_R$ that are the images of the unique vertex of $\rel P$ under $\epsilon_{i,R}$ for the respective $i$'s.
Note that some of the $x_i$'s might coincide, and $\rel Q_R$ can also have other vertices.
Naturally, the elements $x_1, \dots, x_{\ar R}$ depend on the symbol $R$ which will be always clear from the context.
This means that the structure $\Gamma(\rel A)$ can be equivalently described in the following way: the universe of $\Gamma(\rel A)$ coincides with the universe of $\rel A$, and for every $\tau$-symbol $R$, we have
\begin{equation} \label{eq:r-gamma}
  R^{\Gamma(\rel A)} = \{ (r(x_1),\dots, r(x_{\ar R})) \mid r\colon \rel Q_R \to \rel A \}.
\end{equation}
Finally, in this case we have that every homomorphism $h\colon \rel A \to \rel B$ is also a homomorphism $h\colon \Gamma(\rel A) \to \Gamma(\rel B)$.

We will construct the right adjoint $\Omega$ in two steps: First, we define the set of vertices of a structure $\Omega(\rel B)$, and second, we define the edges on these vertices. Before defining edges, we introduce new notation. 

Throughout these definitions, we fix the following setting:
Fix a $(\sigma, \tau)$-Pultr template with $\rel P$ being a singleton structure with empty relations and with $\rel Q_R$ being a $\sigma$-tree for each $\tau$-symbol $R$. First, for each $R$, we pick a term $t_R$ representing $\rel Q_R$, and define $\trees$ to be the set of all subterms of any of the $t_R$'s. We use notation $\trees_V$ and $\trees_S$ (where $S$ is a $\sigma$-symbol) for the $V$-terms and $S$-terms, respectively, that belong to this set.

\begin{definition}[Vertices of $\Omega(\rel B)$]
  \label{def:vertices-of-adjoint-1}
  Let $\rel B$ be a $\tau$-structure. We define $\Omega(B)$ to be the set of all tuples
  \[
    U \in \prod_{t\in \trees_V} \subsets{\hom(\Gamma(\relT t), \rel B)}
  \]
  such that $U_{\vertex}$ is a singleton set, i.e., vertices of $\Omega(\rel B)$ are tuples $U$ indexed by $V$-terms in $\trees$, where the $t$-th entry is a set of homomorphisms from $\Gamma(\relT t)$ to $\rel B$ such that $U_\vertex$ contains exactly one homomorphism.
\end{definition}

Let us remark on the above definition. Observe that the domain of $\relT \vertex$ is $\{r_\vertex\}$, hence $U_\vertex = \{ r_\vertex \mapsto u \}$ for some $u \in B$, i.e., $U_\vertex$ selects an element $u$ of $B$.
Also note that the relations in $\Gamma(\relT t)$ will be empty whenever none of the trees $\rel Q_R$ map to $\relT t$, in which case $U_t$ is a set of functions from $T(t)$ to $B$.

Intuitively, elements of $\Omega(B)$ are elements $u$ of $B$ together with some extra information.
The additional information stores data about mappings from the subtrees of each $\rel Q_R$ to $B$ and how they interact.

To define edges, we need to introduce some notation.
Let $S$ be a relational symbol of arity $k$, $t_1, \dots, t_k \in \trees_V$, and $U^1, \dots, U^k \in \Omega(B)$.
Consider the term $t = \edge_S(t_1, \dots, t_k)$ and the tree $\relT t$. 
Since its domain $T(t)$ is the disjoint union of $T(t_1)$, \dots, $T(t_k)$, 
we may associate to every tuple $(f_1,\dots,f_k)\in U^1_{t_1}\times \dots \times U^k_{t_k}$ a unique mapping $f \colon T(t) \to B$ defined as $f = f_1 \cup \dots \cup f_k$.
We define $\e_t(U^1, \dots, U^k)$ as the set of all such mappings, i.e., we let
\[
  \e_t(U^1, \dots, U^k) = \{f_1 \cup \dots \cup f_k \mid f_1\in U^1_{t_1},\dots,f_k\in U^k_{t_k}\}.
\]
Note that $\e_t(U^1, \dots, U^k)$ is bijective to $U^1_{t_1} \times \dots \times U^k_{t_k}$.
Often, we will write simply $\e_t(E)$ instead of $\e_t(U^1, \dots, U^k)$ if $E = (U^1, \dots, U^k)$.

\begin{definition}[Edges of $\Omega(\rel B)$] \label{def:adjoint-1}
  Let $\rel B$ be a $\tau$-structure, we define a $\sigma$-structure $\Omega(\rel B)$ with universe $\Omega(B)$.
  For each $\sigma$-symbol $S$ of arity $k$, let $S^{\Omega(\rel B)}$ consist of all tuples $(U^1, \dots, U^k) \in \Omega(B)^k$ that satisfy:
  \begin{enumerate}
    \item[\normalfont (A1)] For all $t\in \trees_S$ and $i\in \{1,\dots, k\}$ such that $\pr_i(t) \in \trees_V$, we have
    \[
      \e_t(U^1, \dots, U^k) \subseteq U^i_{\pr_i(t)}.
    \]
    \item[\normalfont (A2)] For all $t\in \trees_S$,
    \[
      \e_t(U^1, \dots, U^k) \subseteq \hom(\Gamma(\relT t), \rel B).
    \]
  \end{enumerate}
\end{definition}

Let us comment on the definition.
In (A1), the trees represented by $t$ and $\pr_i(t)$ coincide, and hence both sets in the condition consist of mappings with the same domain. Item (A2) requires that the mappings that we defined by taking union of homomorphisms $\relT{t_i} \to U^i$, where $t = \edge_R(t_1, \dots, t_k)$, are actually homomorphisms.
Consequently, given that $E$ is an edge in $S^{\Omega(\rel B)}$, we may thus define a tuple
\[
  E^* \in \prod_{t\in \trees_S} \subsets{\hom(\Gamma(\relT t), \rel B)}
\]
with $E^*_t = \e_t(E)$.
The intuition behind requiring condition (A2) is that the union of homomorphisms is a homomorphism as long as the root edge of $t$ is mapped to an edge.
Observe that condition (A2) does not need to be checked if $\pr_i(t) \in \trees_V$ for some $i$ since then it follows from (A1). Nevertheless, it will be useful to talk about it separately.

We claim that this construction $\Omega$ yields a right adjoint to the considered cases of central Pultr functors $\Gamma$.

\begin{theorem} \label{thm:adjoint-1}
  Assume a $(\sigma, \tau)$-Pultr template with $\rel P$ being the $\sigma$-structure with a single vertex and empty relations, and $\rel Q_R$ being a $\sigma$-tree for all $\tau$-symbols $R$.
  Further, assume $\Gamma$ is the central Pultr functor defined by this template, and $\Omega$ is defined as in Definition~\ref{def:adjoint-1}.

  For every $\sigma$-structure $\rel A$ and $\tau$-structure $\rel B$, there is a homomorphism $\Gamma(\rel A) \to \rel B$ if and only if there is a homomorphism $\rel A \to \Omega(\rel B)$.
\end{theorem}

We now proceed to prove the above theorem in several steps. The following lemma proves one of the implications and gives further insights to why $U$'s and $\e_t(E)$'s are defined as above.

\begin{lemma} \label{lem:A-easy}
  If there is a homomorphism $f\colon \Gamma(\rel A) \to \rel B$, then there is a homomorphism $g\colon \rel A \to \Omega(\rel B)$.
\end{lemma}

\begin{proof}
  We define a mapping $g\colon A \to \Omega(B)$ by
  \[
    g(u)_t = \{ f\circ h \mid h\colon \relT t \to \rel A, h(r_t) = u \}.
  \]
  We claim that this mapping is a homomorphism. First, we show that $g(u)$ is well-defined, i.e., that $g(u)_{\vertex}$ is a singleton set, and the elements of $g(u)_t$ are homomorphisms from $\Gamma(\relT{t})$ to $\rel B$. For the former, observe that $g(u)_{\vertex} = \{ r_{\vertex} \mapsto f(u) \}$, since there is a single homomorphism $h\colon \rel T(\vertex) \to \rel A$ which maps the root (and the only vertex) $r_{\vertex}$ to $u$. For the latter, assume $t$ is a $V$-term. Since $h\colon \relT t \to \rel A$ is also a homomorphism $\Gamma(\relT t) \to \Gamma(\rel A)$, and $f$ is a homomorphism $\Gamma(\rel A) \to \rel B$, we get that $f\circ h\colon \Gamma(\relT t) \to \rel B$ is  homomorphism by composition.

  To prove that $g$ preserves the relations, assume that $e = (u_1, \dots, u_k) \in S^{\rel A}$ is an edge. We first show that
  \[
    \e_t(g^S(e)) = \{ f\circ h \mid h\colon \relT t \to \rel A, h^S(r_t) = e \}
  \]
  for each $t\in \trees_S$, $t = \edge_S(t_1, \dots, t_k)$.
  This is true because, for any homomorphism $h\colon \relT t \to \rel A$ that maps the root edge to $e = (u_1, \dots, u_k)$, its restriction $h_i\colon \relT{t_i} \to \rel A$ maps the roots to the respective $u_i$ for all $i$, and, for any tuple of homomorphisms $h_i\colon \relT{t_i} \to \rel A$ which maps the roots to the respective $u_i$'s, their union is a homomorphism $\relT t\to \rel A$.

  To prove property (A1), we need to check that $\e_t(g^S(e)) \subseteq g(u_i)_{\pr_i(t)}$. This is easy to see, since any homomorphism $h\colon \relT t \to \rel A$ that maps $r_t$ to $e$ maps $r_{\pr_i(t)}$, which is the $i$-th component of $r_t$, to $u_i$.
  Finally, the property (A2), that each mapping in $\e_t(g^S(e))$ is a homomorphism, is proved in the same way as the analogous statement for $g(u)$, i.e., it follows from the above claim.
\end{proof}

The above lemma concludes one of the implications that we need for the adjunction. We turn to the other implication which we prove in two steps, each provided by one of the following two lemmas.
The first lemma proves the adjunction in the special case when $\rel A = \relT t$ for some $t\in \trees$. This will be used in the proof of the general case.

\begin{lemma} \label{lem:A-hard-pre}
  Let $t\in \trees$ and let $h\colon \relT t \to \Omega(\rel B)$ be a homomorphism.
  Then the mapping $d\colon T(t) \to B$, defined so that $d(v) = f(r_\vertex)$ where $f$ is the unique element of $h(v)_{\vertex}$, is a homomorphism $\Gamma(\relT t) \to \rel B$.
\end{lemma}

\begin{proof}
  We prove by induction on $t$ that $d \in h(r_t)_t$ if $t$ is a $V$-term, and $d \in \e_t(h^S(r_t))$ if $t$ is an $S$-term.
  \begin{description}
    \item [Case $t = \vertex$] This is a trivial case.
    \item [Case $t = \edge_S(t_1, \dots, t_k)$] Note that restrictions of $h$ to subtrees $\relT{t_i}$ are homomorphisms, so we know that, for all $i$, $h(r_{t_i})_{t_i}$ contains the restrictions of $d$ by the inductive assumption.
    The claim then immediately follows from the definition of $\e_t(h^S(r_t))$.
    \item [Case $t = \pr_i(t')$] Since $h$ is a homomorphism from $\relT t = \relT{t'}$ to $\rel B$, we know that $d\in \e_{t'}(h^S(r_{t'}))$, and the claim subsequently follows by (A1).
  \end{description}
  The lemma then immediately follows either by the definition, if $t$ is a $V$-term, or by (A2), otherwise.
\end{proof}

Note that the induction in the above proof alternates between $V$-terms and $S$-terms. We use the above lemma to prove the general case.

\begin{lemma} \label{lem:A-hard}
  If $g\colon \rel A \to \Omega(\rel B)$ is a homomorphism, then there is a homomorphism $f\colon \Gamma(\rel A) \to \rel B$.
\end{lemma}

\begin{proof}
  Recall that, for each $\tau$-symbol $R$, $t_R$ is a fixed term representing $\rel Q_R$, and hence $\relT{t_R}$ and $\rel Q_R$ are isomorphic. We further assume (without loss of generality) that $\rel Q_R = \relT {t_R}$.
  We define $f$ by setting $f(u)$ to be the unique value attained by the single map in $g(u)_{\vertex}$.
  This is a well-defined mapping on the vertices of $\Gamma(\rel A)$. We need to show that it preserves the relations of $\Gamma(\rel A)$. To this end, assume that $R$ is a $\tau$-symbol of arity $k$ and $(u_1, \dots, u_k) \in R^{\Gamma(\rel A)}$. This means that there is a homomorphism
  \(
    h\colon \rel Q_R \to \rel A,
  \)
  s.t., $h(x_i) = u_i$ for all $i \in [k]$. Observe that
  \(
    g\circ h \colon \relT{t_R} \to \Omega(\rel B)
  \)
  is a homomorphism since it is obtained as a composition of two homomorphisms, so Lemma~\ref{lem:A-hard-pre} applies to $g\circ h$ in place of $h$ and $f\circ h$ in place of $d$ (since $fh(v)$ is the unique value attained by the single map in $gh(v)_{\vertex}$). Consequently, $f\circ h$ is a homomorphism from $\Gamma(\relT{t_R})$ to $\rel B$, and therefore
  \[
    (f(u_1), \dots, f(u_k)) = (fh(x_1), \dots, fh(x_k)) \in R^{\rel B},
  \]
  since $h(x_i) = u_i$ and $(x_1, \dots, x_k) \in R^{\Gamma(\relT{t_R})}$, where the latter follows from the definition of $R^{\Gamma(\relT{t_R})}$ as phrased in \eqref{eq:r-gamma} by letting $r$ be the identity homomorphism on $\relT{t_R}$.
\end{proof}

Lemmas \ref{lem:A-easy} and \ref{lem:A-hard} together yield Theorem~\ref{thm:adjoint-1}.

\begin{remark}
  Let us briefly compare our general construction to that of Foniok and Tardif \cite[Theorem 7.1]{FT15} assuming that both signatures are digraphs and $\rel P = \rel V_1$ (denoted by $\vec P_0$ in \cite{FT15}). An analogous comparison also applies for the construction presented in the next section. Also, examples below compare the two constructions in several concrete cases of central Pultr functors.

  First, both the domain of our right adjoint $\Omega$ and Foniok-Tardif right adjoint $\Omega_\text{FT}$ consist of some tuples indexed by some rooted subtrees of $\rel Q_E$. The first difference is in how we choose the subtrees. Foniok and Tardif choose a vertex of $\rel Q_E$ called \emph{middle vertex} which disconnects the two copies of $\rel P$ in $\rel Q_E$, and define the set of subtrees using this vertex. This is similar to choosing a term representing $\rel Q_E$ rooted in the middle vertex, although not every subtree used by Foniok and Tardif needs to be represented by a subterm.
  Furthermore, note that the requirement that the middle vertex separates the two copies of $\rel P$ creates an obstacle to generalising Foniok and Tardif's construction to relational structures with higher arity.

  Second, the elements of tuples comprising the universe of $\Omega_\text{FT}$ are subsets of vertices of $\rel B$ while elements of our tuples are sets of mappings to $B$. The main reason that Foniok and Tardif are able to do this is that the middle vertex separates the two copies of $\rel P$, and hence they only need to track the value of the mappings on one distinguished vertex. Again, since our aim is to provide adjoints in a more general setting, we cannot afford to do that.
\end{remark}

\subsection{Example: An oriented path}

In this example, we compare our construction to the construction introduced in \cite[Definition 4.1]{FT15}. Our goal is to construct the adjoint to the digraph Pultr functor $\Gamma$ defined by the Pultr template where $\rel Q_E$ is the following digraph:
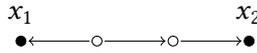
\[
  \begin{tikzpicture}[vertex/.style = {circle, inner sep=1.3, outer sep=1},
      baseline = {([yshift=-.8ex]current bounding box.center)}]
    \node [draw, vertex, fill, label={above:{$x_1$}}] (v0) at (0,0) {};
    \node [draw, vertex] (v1) at (1,0) {};
    \node [draw, vertex] (v2) at (2,0) {};
    \node [draw, vertex, fill, label={above:{$x_2$}}] (v3) at (3,0) {};
    \draw [->] (v1) -- (v0);
    \draw [->] (v1) -- (v2);
    \draw [->] (v2) -- (v3);
  \end{tikzpicture}
\]
The maps $\epsilon_{1,E}$ and $\epsilon_{2,E}$ map the singleton $\rel P$ to $x_1$ and $x_2$, respectively.

Let us start by fixing a term $t_E$ representing $\rel Q_E$. Namely, we let
\[
  t_E = \edge_E(\pr_1(\edge_E(\vertex, \vertex)), \pr_1(\edge_E(\vertex, \vertex)))
\]
which represents $\rel Q_E$ rooted in the middle edge. It has two $V$-subterms and two $E$-subterms (including itself) that represent the following trees:
\begin{align*}
  s_0 = \vertex &\quad
  \begin{tikzpicture}[every node/.style = {circle, inner sep=1.3, outer sep=1},
      baseline = {([yshift=-.8ex]current bounding box.center)}]
    \node [draw=red, fill=red] (v) at (0,0) {};
  \end{tikzpicture} &
  t_1 = \edge_E(s_0, s_0) &\quad
  \begin{tikzpicture}[every node/.style = {circle, inner sep=1.3, outer sep=1},
      baseline = {([yshift=-.8ex]current bounding box.center)}]
    \node [draw] (v0) at (0,0) {};
    \node [draw] (v1) at (1,0) {};
    \draw [->, red, very thick] (v0) -- (v1);
  \end{tikzpicture} \\
  s_1 = \pr_1(t_1) &\quad
  \begin{tikzpicture}[every node/.style = {circle, inner sep=1.3, outer sep=1},
      baseline = {([yshift=-.8ex]current bounding box.center)}]
    \node [draw=red, fill=red] (v0) at (0,0) {};
    \node [draw] (v1) at (1,0) {};
    \draw [->] (v0) -- (v1);
  \end{tikzpicture} &
  t_E = \edge_E(s_1, s_1) &\quad
  \begin{tikzpicture}[every node/.style = {circle, inner sep=1.3, outer sep=1},
      baseline = {([yshift=-.8ex]current bounding box.center)}]
    \node [draw] (v0) at (0,0) {};
    \node [draw] (v1) at (1,0) {};
    \node [draw] (v2) at (2,0) {};
    \node [draw] (v3) at (3,0) {};
    \draw [->] (v1) -- (v0);
    \draw [->, red, very thick] (v1) -- (v2);
    \draw [->] (v2) -- (v3);
  \end{tikzpicture}
\end{align*}
For a directed graph $\rel H$, the definition of $\Omega(\rel H)$ is spelled out as follows. The vertices of $\Omega(\rel H)$ are pairs $U=(U_{s_0}, U_{s_1})$, where $U_{s_0}$ is the set containing the map that sends the unique vertex of $\relT \vertex$ to some $u_0\in H$ and $U_{s_1} \subseteq H^{T(s_1)}$; this is because $\Gamma(\rel T(s_1))$ has no edges since $\rel Q_E \not\to \rel T(s_1)$.
There is an edge from $U=(U_{s_0}, U_{s_1})$ to $V=(V_{s_0}, V_{s_1})$ if
\begin{enumerate}
  \item $\e_{t_1}(U, V) \subseteq U_{s_1}$, and
  \item $\e_{t_E}(U, V) \subseteq \hom(\Gamma(\rel Q_E), \rel H)$
\end{enumerate}
(the remaining condition (A2) for $t = t_1$ is trivial).
Let us simplify this definition.
First, we will write homomorphisms from the above paths as tuples, writing the values of such homomorphisms from left to right as the vertices appear on the picture above. In this way, we have $U_{s_0} \subseteq H$, $U_{s_1} \subseteq H\times H$ for each $U$.
Since $t_1 = \edge_e(s_0, s_0)$, we get a bijection $\e_{t_1} \simeq U_{s_0} \times V_{s_0}$. Furthermore, using the above ordering, we may simply say that $\e_{t_1} = U_{s_0} \times V_{s_0}$. Similarly, $\e_{t_E}(U, V) \simeq U_{s_1} \times V_{s_1}$, since $t_E = \edge_E(s_1, s_1)$, more precisely
\[
  \e_{t_E}(U, V) = \{ (u_1, u_0, v_0, v_1) \mid
    (u_0, u_1) \in U_{s_1} \text{ and } (v_0, v_1) \in V_{s_1} \}.
\]
The two conditions are then rephrased as follows:
\begin{enumerate}
  \item $U_{s_0} \times V_{s_0} \subseteq U_{s_1}$, and
  \item $(u_1, v_1) \in E^{\rel H}$ for every $(u_0, u_1) \in U_{s_1}$ and $(v_0, v_1) \in V_{s_1}$.
\end{enumerate}

We claim that this construction results (on the same input) in a digraph that is homomorphically equivalent to the one obtained by \cite[Definition 4.1]{FT15}. Namely, the adjoint constructed there, let us call it $\Omega'$, is as follows: The vertices of $\Omega'(\rel H)$ are pairs $(a, A)$, where $a\in H$ and $A\subseteq H$, and there is an edge from $(a, A)$ to $(b, B)$ if $b\in A$ and $A\times B \subseteq E^{\rel H}$.

We show that, for every graph $\rel H$, there is a homomorphism $\alpha\colon \Omega'(\rel H) \to \Omega(\rel H)$ defined by $\alpha((a,A)) = U$ where $U_{s_0} = \{ a \}$ and $U_{s_1} = \{ a \} \times A$.
To show that $\alpha$ preserves edges, assume that $(a, A)$ and $(b, B)$ are connected by an edge in $\Omega'(\rel H)$, i.e., $b\in A$ and $A\times B \subseteq E^{\rel H}$, and $U = \alpha((a,A))$, $V = \alpha((b, B))$. We have that
\begin{align*}
  \e_{t_1}(U, V) &= \{ (a, b) \} \\
  \e_{t_E}(U, V) &= A \times \{ (a, b) \} \times B,
\end{align*}
and claim that:
(1) $\e_{t_1}(U, V) \subseteq U_{s_1}$; which is true because $b\in A$.
(2) $\e_{t_E}(U, V) \subseteq \hom(\Gamma(\rel Q_E), \rel H)$; which is true, since the only edge of $\Gamma(\rel Q_E)$ is $(x_1, x_2)$ and the projection of $\e_{t_E}(U, V)$ to $x_1, x_2$ (the first and the last coordinate) is $A \times B \subseteq E^{\rel H}$.

A homomorphism $\beta\colon \Omega(\rel H) \to \Omega'(\rel H)$ is given by
$\beta(U) = (a, A)$ where $a$ is the unique element of $U_{s_0}$, and
\[
  A = \{ a' \mid (a, a') \in U_{s_1} \}.
\]
To show that it is a homomorphism, assume $(U, V) \in E^{\Omega(\rel H)}$, and let $\beta(U) = (a, A)$ and $\beta(V) = (b, B)$.
Since $\e_{t_1}(U, V) = U_{s_0} \times V_{s_0} = \{(a, b)\}$ and $\e_{t_1}(U, V) \subseteq U_{s_1}$, we have $(a, b) \in U_{s_1}$, which implies that $b \in A$. Also since $\e_{t_E}(U, V) \subseteq \hom(\Gamma(\rel Q_E), \rel H)$, $a\times A \subseteq U_{s_1}$, $b \times B \subseteq V_{s_1}$, and $t_E = \edge_E(s_1, s_1)$, we have that
\[
  A \times \{(a, b)\} \times B \subseteq \e_{t_E}(U, V) \subseteq \hom(\Gamma(\rel Q_E), \rel H),
\]
where the first inclusion follows from the definition of $\e_{t_E}(U, V)$, and the second includion follows from (A2).
In particular, the inclusion above implies that $A \times B \subseteq E^{\rel H}$. This concludes the proof of the homomorphic equivalence of $\Omega(\rel H)$ and $\Omega'(\rel H)$.

We note that our $\Omega(\rel H)$ can be reduced to a smaller homomorphically equivalent structure by requiring that vertices $U \in \Omega(\rel H)$ satisfy
\begin{enumerate}
    \item[(A3)] for all $s, s' \in \trees_V$ and homomorphisms $h\colon \relT s \to \rel T(s')$ with $h(r_s) = r_{s'}$, we have
    \[
      \{ f \circ h \mid f\in U_{s'} \} \subseteq U_s.
    \]
\end{enumerate}
Note that the elements constructed in the proof of Lemma~\ref{lem:A-easy} satisfy this property. In this particular example, this requirement would force that $U_{s_1} = a \times A$ for some $A\subseteq H$ and $a \in U_{s_0}$, since $s_0$ is embedded to $s_1$ as the root. This would then make the two homomorphisms defined above isomorphisms.

\subsection{Example: A 4-ary relation defined by an oriented path}

Our definition works also for Pultr templates that are not just digraph templates. As an example for comparison, let us consider a Pultr template that is similar to the previous example, but in this case maps digraphs to structures over a signature containing one 4-ary relational symbol~$R$.

Specifically, $\rel P$ is still a singleton with no edges, and $\rel Q_R$ is the same digraph as $\rel Q_E$ above, but we now have 4 homomorphisms $\epsilon_{i,R} \colon \rel P \to \rel Q_R$ for $i = 0,1,2,3$ which map the vertex of $\rel P$ to $0, 1, 2$, or $3$ respectively. Pictorially, the digraph $\rel Q_R$ together with its distinguished vertices $x_0$, \dots, $x_3$ is as follows.
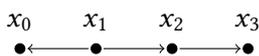
\[
  \begin{tikzpicture}[vertex/.style = {circle, inner sep=1.3, outer sep=1},
      baseline = {([yshift=-.8ex]current bounding box.center)}]
    \node [draw, vertex, fill, label={above:{$x_0$}}] (v0) at (0,0) {};
    \node [draw, vertex, fill, label={above:{$x_1$}}] (v1) at (1,0) {};
    \node [draw, vertex, fill, label={above:{$x_2$}}] (v2) at (2,0) {};
    \node [draw, vertex, fill, label={above:{$x_3$}}] (v3) at (3,0) {};
    \draw [->] (v1) -- (v0);
    \draw [->] (v1) -- (v2);
    \draw [->] (v2) -- (v3);
  \end{tikzpicture}
\]
We let $t_R = t_E$ where $t_E$ is as above, and we use the same notation as in the previous example.

For a structure $\rel B$ with a 4-ary relation $R^{\rel B}$, vertices of $\Omega(\rel B)$ are defined in a similar way as above, i.e., they are pairs $(U_{s_0}, U_{s_1})$ where $U_{s_0} = \{b\}$ for some $b\in B$ and $U_{s_1} \subseteq B\times B$. Two such vertices $(U, V)$ are connected by an edge if
\begin{enumerate}
  \item $\e_{t_R}(U, V) \subseteq R^{\rel B}$, and
  \item $\e_{t_1}(U, V) \subseteq U_{s_1}$;
\end{enumerate}
where $t_R = \edge_E(s_1, s_1)$ and $t_1 = \edge_E(s_0, s_0)$.
More precisely, the first condition can be written as
\[
  \{ (u_1, u_0, v_0, v_1) \mid
    (u_0, u_1) \in U_{s_1}, (v_0, v_1) \in V_{s_1} \}
  \subseteq R^{\rel B},
\]
and the second condition can be written as $U_{s_0} \times V_{s_0} \subseteq U_{s_1}$.

The only real difference from the previous example is that $\e_{t_R}(U, V) \subseteq R^{\rel B}$ instead of requiring that the projection of $\e_{t_E}(U, V)$ on the first and the last coordinates is a subset of $E^{\rel H}$.
As above, we could create a homomorphically equivalent $\Omega'(\rel B)$ whose vertices would be pairs $(u, U)$ where $u\in B$ and $U \subseteq B$. Two such vertices $(u, U)$ and $(v, V)$ would be then connected by an edge if $v \in U$ and
\(
  U \times \{(u,v)\} \times V \subseteq R^{\rel B}
\).

\subsection{Duals from adjoints}
    \label{sec:adjoints-and-duals}

In \cite{FT15}, the authors claim that, for digraph Pultr templates with the edge relation defined by $\rel Q_E$, the image of the digraph with a single vertex and no edges under the right Pultr functor is a dual to $\rel Q_E$.
We may be slightly more precise when talking about our constructions, namely, we claim the following.

\begin{proposition}
  Let $\tau$ contain a single symbol $R$, and fix a $(\sigma, \tau)$-Pultr template defined by $\rel P = \rel V_1$ and $\rel Q_R = \relT{t_R}$ for some term $t_R$ such that $\rel Q_R \not\to \relT t$ for any $t < t_R$ (i.e., $t_R$ is a minimal term among those representing a structure homomorphically equivalent to $\rel Q_R$).
  Let $\Omega$ be the right adjoint to $\Gamma$ as defined in Definition~\ref{def:adjoint-1}.
  Then the image of the $\tau$-structure $\rel V_1$ under $\Omega$ is isomorphic to $\rel D(t_R)$.
\end{proposition}

\begin{proof}
  Observe that there is only one function from any set $T$ to $V_1 = \{1\}$.
  Hence, there is only one candidate function $f$ for a homomorphism $\Gamma(\relT t) \to \rel V_1$ for any term $t$. It is not hard to observe that $f$ is a homomorphism if and only if $\Gamma(\relT t)$ has no edges. In particular, $f$ is a homomorphism for all $t < t_r$, since $\rel Q_R \not\to \relT t$ in that case.
  Consequently, each of the components of an element $U \in \Omega(V_1)$ is either the empty set or a singleton set.
  Furthermore, we have that $\hom(\Gamma(\relT{t_R}), \rel V_1) = \emptyset$, since $\Gamma(\relT{t_R})$ has an edge (witnessed by the identity homomorphism).

  The rest of the proof is based on the idea of treating $\emptyset$ as $\false$ and any singleton set as $\true$.
  With this interpretation, we get that $X \impl Y$ is equivalent to $X \subseteq Y$ and $X \meet Y$ is equivalent to $X \times Y$ whenever $X, Y$ are sets with at most one element. Note that this draws an immediate parallel between the definitions of $\w_t$ and $\e_t$, (D1) and (A1), and also (D2) and (A2) if we take into account that $\hom(\Gamma(\relT{t_R}), \rel V_1) = \emptyset$.

  We may now define an isomorphism $\Omega(\rel V_1) \simeq \rel D(t_R)$ by assigning to $u \in \rel D(t_R)$ an element $U \in \Omega(\rel V_1)$ where,  for each $t\in \trees_V$, $u_t = \false$ if $U_t = \emptyset$, and $u_t = \true$ if $U_t = \{ x\mapsto 1 \}$. The observations in the first paragraph of this proof show that this assignment is bijective. 
  
  To show that it preserves edges, observe that, for each $t\in \trees_R$, $\w_t(u^1, \dots, u^k) = \false$ if and only if $\e_t(U^1, \dots, U^k) = \emptyset$, since the former is defined as
  \(
    \w_t(u_1, \dots, u_k) = u^1_{t_1} \meet \dots \meet u^k_{t_k}
  \)
  and the latter satisfies
  \(
    \e_t(U_1, \dots, U_k) \simeq U^1_{t_1} \times \dots \times U^k_{t_k}
  \).
  Furthermore, (A1) is equivalent to (D1) since, for all $t\in \trees_R$ with $\pr_i(t) \in \trees_V$, $\w_t(u^1, \dots, u^k) \impl u^i_{\pr_i(t)}$ is equivalent to $\e_t(U^1, \dots, U^k) \subseteq U^i_{\pr_i(t)}$.
  Finally, (A2) $\w_{t_Q}(u^1, \dots, u^k) = \false$ is equivalent to
  \begin{equation}
    \text{for all $t\in \trees_R$, } \e_t(U^1, \dots, U^k) \subseteq \hom(\Gamma(\relT t), \rel B).
    \tag{D2}
  \end{equation}
  This is because $\hom(\Gamma(\relT t), \rel B)) = \emptyset$ if and only if $t = t_Q$, hence (D2) is trivially satisfied unless $t = t_Q$, in which case it says $\e_{t_Q}(U^1, \dots, U^k) = \emptyset$.
\end{proof}

\section{Adjoints to functors with domains defined by a relation}
  \label{sec:adjoint-edge}

In this section, we construct a right adjoint to the central Pultr functor defined by a $(\sigma, \tau)$-Pultr template for which $\rel P = \relcheck S_1$, for some choice of $\sigma$-symbol $\check S$, i.e., $\rel P$ is an $\check S$-edge.
Again, for each $k$-ary $\tau$-symbol $R$, we have a $\sigma$-tree $\rel Q_R$ with a $k$-tuple $\epsilon_{1,R}$, \dots, $\epsilon_{k,R}$ of homomorphisms $\rel P \to \rel Q_R$. Each of these homomorphisms selects an $\check S$-edge of $\rel Q_R$, we denote these edges by $x_1$, \dots, $x_k$, respectively, i.e., $x_i = \epsilon_{i,R}^{\check S}(e)$ for each $i = 1, \dots, k$ where $e \in \check S^{\rel P}$ is the unique $\check S$-edge of $\rel P$.
Naturally, the edges $x_1, \dots, x_{\ar R}$ depend on the symbol $R$ which will be always clear from the context.
Let us repeat the definition of $\Gamma(\rel A)$ for a $\sigma$-structure $\rel A$ using this notation:
The domain of $\Gamma(\rel A)$ is $\check S^{\rel A}$, and
for each $\tau$-symbol $R$ of arity $k$, the corresponding relation of $\Gamma(\rel A)$ is defined as
  \[
    R^{\Gamma(\rel A)} =
    \{ (h^{\check S}(x_1), \dots, h^{\check S}(x_k)) \mid h\colon \rel Q_R \to \rel A \}.
  \]
We also note that, for each homomorphism $h\colon \rel A \to \rel B$ between two $\sigma$-structures $\rel A$ and $\rel B$, $h^{\check S}\colon \Gamma(\rel A) \to \Gamma(\rel B)$ is a homomorphism.

The construction of $\Omega$ in this case is almost identical to the construction in Section~\ref{thm:adjoint-1} (Definitions~\ref{def:vertices-of-adjoint-1} and \ref{def:adjoint-1}), and we present it in a similar way.

We fix the following setting:
Fix a $(\sigma, \tau)$-Pultr template with $\rel P = \relcheck S_1$ for some $\sigma$-symbol $\check S$, and assume $\rel Q_R$ is a $\sigma$-tree for each $\tau$-symbol $R$.
For each $\tau$-symbol $R$, we pick a term $t_R$ representing $\rel Q_R$, and we let $\trees$ be the set of all subterms of any of the $t_R$'s. We use notation $\trees_V$ and $\trees_S$ for the $V$-terms and $S$-terms, where $S$ is a $\sigma$-symbol, that belong to this set.

\begin{definition}[Vertices of $\Omega(\rel B)$]
  \label{def:vertices-of-adjoint-edge}
  Let $\rel B$ be a $\tau$-structure. We define $\Omega(B)$ to be the set of all tuples
  \[
    U \in \prod_{t\in \trees_V} \subsets{\hom(\Gamma(\relT t), \rel B)}
  \]
  such that $U_{\vertex}$ is a singleton set, i.e., $U_{\vertex} = \{\emptyset\}$.
\end{definition}

As before, vertices of $\Omega(\rel B)$ are tuples $U$ indexed by $V$-terms in $\trees$, where the $t$-th entry is a set of homomorphisms from $\Gamma(\relT t)$ to $\rel B$ such that $U_\vertex$ contains exactly one homomorphism. A difference here is that this homomorphism is the map $\emptyset \to B$, since $\Gamma(\relT \vertex)$ has no vertices as $\rel P \not\to \relT \vertex$.
Thus $U_\vertex$ does not contain any information; it serves a similar purpose as $u_\vertex = \true$ in Definition~\ref{def:duals}.

To define edges, we use similar notation $\e_t$ as before.
Let $t_1, \dots, t_k \in \trees_V$, $S$ a relational symbol of arity $k$, and $U^1, \dots, U^k \in \Omega(B)$. Again, we consider the term $t = \edge_S(t_1, \dots, t_k)$ and the tree $\relT t$. Unlike in Section 5, the domain of $\Gamma(\relT t)$ is now $\check S^{\relT t}$ and not $T(t)$. This means we have to distinguish two cases:

\paragraph{Case 1: $S = \check S$}
  In this case, the domain of $\Gamma(\relT t)$ is
  \[
    \check S^{\relT t} =
      \check S^\relT{t_1} \cup \dots \cup \check S^\relT{t_k} \cup \{r_t\}.
  \]
  Hence, in order to define a mapping $f \colon \check S^{\relT t} \to B$, we need to specify its value on $r_t$ in addition to its restrictions to $S^{\relT{T_i}}$'s. In detail, a $k$-tuple of mappings $f_1\in U^1_{t_1}$, \dots, $f_k \in U^k_{t_k}$ together with an element $e_\bullet \in B$ uniquely defines a mapping $f \colon \check S^{\relT t} \to B$ by
  \[
    f = f_1 \cup \dots \cup f_k \cup \{ r_t \mapsto e_\bullet \}.
  \]
  Therefore, for each $e_\bullet \in B$, we denote by $\e_t(U^1, \dots, U_k; e_\bullet)$ the set of all such mappings. Again, we have that, for each $e_\bullet$, $\e_t(U^1, \dots, U_k; e_\bullet)$ is bijective to $U^1\times \dots \times U^k$.
  We will simply write $\e_t(E; e_\bullet)$ for $\e_t(U^1, \dots, U^k; e_\bullet)$ if $E = (U^1, \dots, U^k)$.

\paragraph{Case 2: $S \neq \check S$}
  In this case, we have $\check S^{\relT t} = \check S^\relT{t_1} \cup \dots \cup \check S^\relT{t_k}$. Thus the domain of $\Gamma(\relT t)$ is the disjoint union of domains of $\Gamma(\relT{t_i})$, and we define $\e_t(U^1, \dots, U^k)$ the same way as before, i.e., as the set of all unions $f_1 \cup \dots \cup f_k$ where $f_i \in U^i_{t_i}$ for each $i$.

If $S \neq \check S$, $S$-edges are defined analogously to Definition~\ref{def:adjoint-1}. To define $\check S$-edges, we have to take into account the new element $e_\bullet$.

\begin{definition}[Edges of $\Omega(\rel B)$] \label{def:adjoint-edge}
  Let $\rel B$ be a $\tau$-structure, we define a $\sigma$-structure $\Omega(\rel B)$ with universe $\Omega(B)$.

  We first define $\check S$-edges. Let $k$ be the arity of $\check S$.
  $\check S^{\Omega(\rel B)}$ consists of all tuples $(U^1, \dots, U^k) \in \Omega(B)^k$ for which there exists $e_\bullet \in B$ such that
  \begin{enumerate}
    \item[\normalfont (B1)] For all $t\in \trees_S$ and $i\in \{1,\dots, k\}$ such that $\pr_i(t) \in \trees_V$, we have
    \[
      \e_t(U^1, \dots, U^k; e_\bullet) \subseteq U^i_{\pr_i(t)}.
    \]
    \item[\normalfont (B2)] For all $t\in \trees_S$,
    \[
      \e_t(U^1, \dots, U^k; e_\bullet) \subseteq \hom(\Gamma(\relT t), \rel B).
    \]
  \end{enumerate}

  If $S \neq \check S$, the relation $S^{\Omega(\rel B)}$ is defined in the same way as in Definition~\ref{def:adjoint-1} except we use the meaning of $\e_t$ defined in this section.
\end{definition}

Given an edge $E \in \check S^{\Omega(\rel B)}$, we call the element $e_\bullet$, which satisfies the conditions (B1) and (B2), a \emph{witness} of this edge. Note that this witness is the only significant difference between Definitions~\ref{def:adjoint-1} and \ref{def:adjoint-edge}.

We claim that this definition indeed constructs a right adjoint to $\Gamma$.

\begin{theorem} \label{thm:adjoint-edge}
  Assume a $(\sigma, \tau)$-Pultr template with $\rel P$ being the $\sigma$-tree with $\ar \check S$ vertices connected by an $\check S$-edge for some $\sigma$-symbol $\check S$, and $\rel Q_R$ being a $\sigma$-tree for all $\tau$-symbols $R$.
  Further, assume $\Gamma$ is the central Pultr functor defined by this template, and $\Omega$ is defined as in Definition~\ref{def:adjoint-edge}.

  For every $\sigma$-structure $\rel A$ and $\tau$-structure $\rel B$, there is a homomorphism $\Gamma(\rel A) \to \rel B$ if and only if there is a homomorphism $\rel A \to \Omega(\rel B)$.
\end{theorem}

The proof is analogous to the proof of Theorem~\ref{thm:adjoint-1} with the following changes: we use $h^{\check S}$ in place of $h$ whenever $h\colon \rel A \to \rel B$ was used as a homomorphism $\Gamma(\rel A) \to \Gamma(\rel B)$. Furthermore, we use a witness $e_\bullet$ in place of the unique value of the homomorphism $f\in U_\vertex$ throughout the proof. In particular, if $f\colon \rel A \to \Omega(\rel B)$, we define $g\colon \Gamma(\rel A) \to \rel B$ by letting $g(e)$ be the witness of the edge $f^{\check S}(e)$. With these substitutions all the arguments of the previous section apply in the case of Theorem~\ref{thm:adjoint-edge}. For completeness and reference, we include the proof in full detail in the last subsection of this section.

\subsection{Example: The arc graph construction}
    \label{sec:arc-graph}

Recall the arc-graph construction from Example~\ref{ex:arc-graph}, which can be expressed as the central Pultr functor whose template consists of structures $\rel P = (\{0,1\}; \{(0,1)\})$ and $\rel Q_E = (\{0,1,2\}; \{(0,1), (1,2)\})$ with $\epsilon_1(0) = 0$ and $\epsilon_1(1) = 1$, and $\epsilon_2(0) = 1$ and $\epsilon_2(1) = 2$.

The right adjoint $\Omega$ according to our definition above would be constructed in the following way: First, we choose a term $t_E$ representing $\rel Q_E$. We can pick $t_2$ as in Section~\ref{ex:path-dual}.
\begin{gather*}
  \begin{tikzpicture}[every node/.style = {circle, inner sep=1.3, outer sep=1},
      baseline = {([yshift=-.8ex]current bounding box.center)}]
    \node [draw] (b1) at (1,0) {};
    \node [draw] (b2) at (2,0) {};
    \node [draw] (b3) at (3,0) {};
    \draw [->] (b1) -- (b2);
    \draw [->, very thick, red] (b2) -- (b3);
  \end{tikzpicture}
  \\
  t_2 = \edge_E(\pr_2(\edge_E(\vertex, \vertex)), \vertex)
\end{gather*}
We also name all its subterms as in Section~\ref{ex:path-dual}, i.e.,
\begin{align*}
  s_0 &= \vertex      & t_1 &= \edge_E(s_0, s_0) \\
  s_1 &= \pr_2 (t_1)
\end{align*}

The vertices of $\Omega(\rel B)$ are defined as pairs $(U_{s_0}, U_{s_1})$ such that $U_{s_0} = \{\emptyset\}$ and $U_{s_1} \subseteq \hom(\Gamma(\relT{s_1}), \rel B)$.
Two such vertices $U, V$ are connected by an edge if there is $e_\bullet\in B$ such that
\begin{enumerate}
  \item $\e_{t_1}(U, V; e_\bullet) \subseteq V_{s_1}$ --- this is condition (B1) for $t = t_1$;
  \item $\e_{t_i}(U, V; e_\bullet) \subseteq \hom(\Gamma(\rel T(t_i)), \rel B)$ for $i = 1,2$ --- this is condition (B2).
\end{enumerate}

These conditions can be considerably simplified.
First, since $U_{s_0} = V_{s_0} = \{ \emptyset \}$, we have
\begin{align*}
  \e_{t_1}(U, V; e_\bullet) &= \{ r_{t_1} \mapsto e_\bullet \} \\
  \e_{t_2}(U, V; e_\bullet) &= \{ f \union (r_{t_2} \mapsto e_\bullet) \mid f\in U_{s_1} \}.
\end{align*}
Further, since $\Gamma(\rel T(s_1)) = \Gamma(\rel T(t_1))$ is the graph with a single vertex and no edges, we can identify $U_{s_1}$ and $\e_{t_1}(U, V; e_\bullet)$ with subsets of $B$. Connecting this observation with the comment about $U_{s_0}$, we can identify $U$ with its coordinate $U_{s_1} \subseteq B$.
Finally, if $e_\bullet$ is a witness, the elements of $\e_{t_2}(U, V; e_\bullet)$ are homomorphisms from a directed edge to $\rel B$ which correspond to the edges of $\rel B$. Using this correspondence, we can identify $\e_{t_2}(U, V; e_\bullet)$ with a subset of $E^{\rel B}$. Taking all of these into account the conditions above simplify to
\begin{enumerate}
  \item $\e_{t_1}(U, V; e_\bullet) = \{ e_\bullet \} \subseteq V$, and
  \item $\e_{t_2}(U, V; e_\bullet) = U \times \{ e_\bullet \} \subseteq E^{\rel B}$.
\end{enumerate}
So, we can say that a vertex of $\Omega(\rel B)$ is a subset $U$ of $B$, and $(U, V)$ is an edge of $\Omega(\rel B)$ if
\[
  \exists e_\bullet \in V
  \text{ such that }
  U\times \{ e_\bullet \} \subseteq E^{\rel B}.
\]

We compare this construction with the functor $\delta_R$ described in \cite[Definition 3.1]{FT15} as a right adjoint to $\delta$. For a digraph $\rel B$, the vertices of the digraph $\delta_R(\rel B)$ are the complete bipartite subgraphs of $\rel B$, i.e., pairs $(U^-,U^+)$ of subsets of vertices of $\rel B$ such that $U^- \times U^+ \subseteq E^{\rel B}$. There is an edge from $(U^-, U^+)$ to $(V^-, V^+)$ if $U^+ \cap V^- \neq \emptyset$. Below, we show that $\delta_R(\rel B)$ and $\Omega(\rel B)$ are homomorphically equivalent.

We start by constructing a homomorphism $h\colon \delta_R(\rel B) \to \Omega(\rel B)$. We let
\[
  h(U^-, U^+) = U^-.
\]
To show that it preserves edges, assume $U^+\cap V^- \neq \emptyset$, i.e., there exists $e_\bullet \in U^+ \cap V^-$. We claim that this $e_\bullet$ witnesses that $U^-$ and $V^-$ is an edge in $\Omega(\rel B)$. Clearly, $e_\bullet \in V^-$. Also, we have
\[
  U^- \times \{e_\bullet\} \subseteq U^- \times U^+ \subseteq E^{\rel B}.
\]

A homomorphism $g\colon\Omega(\rel B) \to \delta_R(\rel B)$ is a bit harder to construct. Guided by the above, it is natural to choose the first component of $g(U)$ to be $g(U)^- = U$. We need to define the second component $g(U)^+$.
We let $g(U)^+$ be the largest set such that $g(U)^- \times g(U)^+ \subseteq E^{\rel B}$, i.e.,
\[
  g(U)^+ = \{ v \in B \mid \forall u\in U, (u, v)\in E^{\rel B} \}.
\]
Now, assume that $U$ and $V$ are connected by an edge in $\Omega(\rel B)$ witnessed by $e_\bullet$. We claim that $e_\bullet \in g(U)^+ \cap g(V)^-$. By definition of $\Omega(\rel B)$, we have $e_\bullet \in V = g(V)^-$, and $U \times \{e_\bullet\} \subseteq E^{\rel B}$, which implies that $e_\bullet \in g(U)^+$. Altogether, $e_\bullet \in g(U)^+ \cap g(V)^-$, and hence $(g(U),g(V)) \in E^{\delta_R(\rel B)}$. This completes the proof.

This example shows how $U_{\vertex}$ can be eliminated from the definition of $\Omega(\rel B)$. Let us repeat again, that the only purpose of $U_{\vertex}$ is to avoid case distinction between some $E$-terms, e.g., between $\edge_E(t, s)$ and $\edge_E(t, \vertex)$ for $s\neq \vertex$. We may ignore it in this example, since we expanded every single case when it is used.

\subsection{Example: Arc structure}
  \label{sec:arc-structure}

In this subsection, we consider a certain variant of the arc graph construction, which we will call an \emph{arc structure} and which encodes more information than the arc graph: The domain of the arc structure coincides with the domain of the arc graph, i.e., the set of all edges of the input graph, and we extend the signature with two more binary symbols that will relate those pairs of edges that are incident in a different sense.
The goal of this example is two-fold: first, to show how the construction of right adjoints works in a more general signature, and second, to show how the right adjoint changes if we change the central Pultr functor in such a way that it encodes more information about the input structure. Note, for example, that any tree can be recovered from its image under the arc structure construction, but there are trees (e.g., the one presented in Fig.~\ref{fig:arc-structure}) that cannot be recovered from their arc graph.

We fix $\phi$ to be a signature with three binary relations $D$, $I$, and $O$, and we let $\gamma$ be the signature of digraphs. We define a central Pultr functor $\bd$ using the $(\gamma, \phi)$-Pultr template defined as follows:
The digraph defining vertices is the digraph with a single directed edge, i.e., $\rel P = \rel E_1$, and the digraphs $\rel Q_D$, $\rel Q_I$, $\rel Q_O$ are the following where the images of $\epsilon_i$'s are highlighted and labelled by $x_i$.
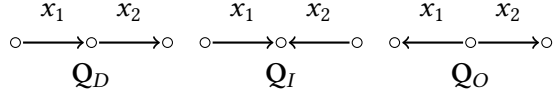
\[
  \begin{matrix}
    \begin{tikzpicture}[every node/.style = {circle, inner sep=1.3, outer sep=1}]
        baseline = {([yshift=-.8ex]current bounding box.center)}]
      \node [draw] (b1) at (1,0) {};
      \node [draw] (b2) at (2,0) {};
      \node [draw] (b3) at (3,0) {};
      \draw [->,thick] (b1) -- (b2) node [midway, label={above:$x_1$}] {};
      \draw [->,thick] (b2) -- (b3) node [midway, label={above:$x_2$}] {};
    \end{tikzpicture}
    &
    \begin{tikzpicture}[every node/.style = {circle, inner sep=1.3, outer sep=1}]
      \node [draw] (b1) at (1,0) {};
      \node [draw] (b2) at (2,0) {};
      \node [draw] (b3) at (3,0) {};
      \draw [->,thick] (b1) -- (b2) node [midway, label={above:$x_1$}] {};
      \draw [<-,thick] (b2) -- (b3) node [midway, label={above:$x_2$}] {};
    \end{tikzpicture}
    &
    \begin{tikzpicture}[every node/.style = {circle, inner sep=1.3, outer sep=1}]
      \node [draw] (b1) at (1,0) {};
      \node [draw] (b2) at (2,0) {};
      \node [draw] (b3) at (3,0) {};
      \draw [<-,thick] (b1) -- (b2) node [midway, label={above:$x_1$}] {};
      \draw [->,thick] (b2) -- (b3) node [midway, label={above:$x_2$}] {};
    \end{tikzpicture} \\
    \rel Q_D & \rel Q_I & \rel Q_O
  \end{matrix}
\]
Note that the digraph $\rel Q_D$ with the two distinguished edges defines the arc-graph functor. This means that for each digraph $\rel G$, the reduct $(A, D^{\rel A})$ where $\rel A = \bd(\rel G)$ is the arc-graph of $\rel G$ (if $D$ is interpreted as $E$), see Figure~\ref{fig:arc-structure}.
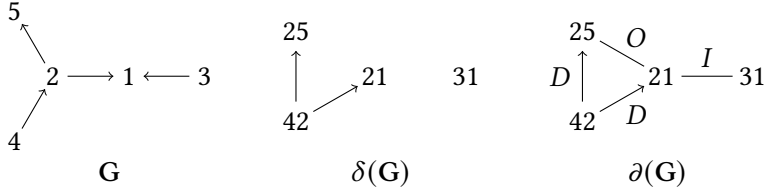
\begin{figure}
  \[\begin{matrix}
    \begin{tikzpicture}[every node/.style = {circle, inner sep=1, outer sep=0},
        baseline = {([yshift=-.8ex]current bounding box.center)}]
      \node (1) at (1,0) {1};
      \node (2) at (0,0) {2};
      \node (3) at (2,0) {3};
      \node (4) at (-.5,-.866) {4};
      \node (5) at (-.5,.866) {5};
      \draw [->] (2) -- (1);
      \draw [->] (3) -- (1);
      \draw [->] (4) -- (2);
      \draw [->] (2) -- (5);
    \end{tikzpicture} &\quad
    \begin{tikzpicture}[every node/.style = {circle, inner sep=1, outer sep=0},
        baseline = {([yshift=-.8ex]current bounding box.center)}, scale = 1.2]
      \node (42) at (-.866,-.5) {42};
      \node (25) at (-.866,.5) {25};
      \node (21) at (0,0) {21};
      \node (31) at (1,0) {31};
      \draw [->] (42) -- (25);
      \draw [->] (42) -- (21);
    \end{tikzpicture} \quad&
    \begin{tikzpicture}[every node/.style = {circle, inner sep=1, outer sep=0},
        baseline = {([yshift=-.8ex]current bounding box.center)}, scale = 1.2]
      \node (42) at (-.866,-.5) {42};
      \node (25) at (-.866,.5) {25};
      \node (21) at (0,0) {21};
      \node (31) at (1,0) {31};
      \draw [->] (42) -- (25) node [midway, label={left:$D$}] {};
      \draw [->] (42) -- (21) node [midway, label={-60:$D$}] {};
      \draw (25) -- (21) node [midway, label={60:$O$}] {};
      \draw (31) -- (21) node [midway, label={above:$I$}] {};
    \end{tikzpicture}
    \\
    \rel G & \delta(\rel G) & \bd(\rel G)
  \end{matrix}\]
    \caption{A digraph, its arc graph and its arc structure. The $O$ and $I$ relations of $\bd(\rel G)$ are symmetric, and $O$ and $I$ loops on all vertices of $\bd(\rel G)$ are omitted for readability.}
    \label{fig:arc-structure}
\end{figure}

Now, to construct the right adjoint to $\bd$, which we denote by $\omega_\bd$, we fix the following terms representing the graphs $\rel Q_D$, $\rel Q_I$, and $\rel Q_O$.
\[
  \begin{matrix}
    \begin{tikzpicture}[every node/.style = {circle, inner sep=1.3, outer sep=1},
        baseline = {([yshift=-.8ex]current bounding box.center)}]
      \node [draw] (b1) at (1,0) {};
      \node [draw] (b2) at (2,0) {};
      \node [draw] (b3) at (3,0) {};
      \draw [->] (b1) -- (b2) {};
      \draw [->,thick,red] (b2) -- (b3) {};
    \end{tikzpicture}
    &
    \begin{tikzpicture}[every node/.style = {circle, inner sep=1.3, outer sep=1},
        baseline = {([yshift=-.8ex]current bounding box.center)}]
      \node [draw] (b1) at (1,0) {};
      \node [draw] (b2) at (2,0) {};
      \node [draw] (b3) at (3,0) {};
      \draw [->] (b1) -- (b2) {};
      \draw [<-,thick,red] (b2) -- (b3) {};
    \end{tikzpicture}
    &
    \begin{tikzpicture}[every node/.style = {circle, inner sep=1.3, outer sep=1},
        baseline = {([yshift=-.8ex]current bounding box.center)}]
      \node [draw] (b1) at (1,0) {};
      \node [draw] (b2) at (2,0) {};
      \node [draw] (b3) at (3,0) {};
      \draw [<-] (b1) -- (b2) {};
      \draw [->,thick,red] (b2) -- (b3) {};
    \end{tikzpicture} \\
    t_D = \edge_E(s_2, \vertex) &
    t_I = \edge_E(\vertex, s_2) &
    t_O = \edge_E(s_1, \vertex)
  \end{matrix}
\]
where $s_i = \pr_i(\edge_E(\vertex, \vertex))$ for $i=1,2$. We name all remaining subterms as follows, $s_0 = \vertex$, $t_E = \edge_E(\vertex, \vertex)$. The set of terms defining $\omega_\bd$ is $\trees = \{ s_0, s_1, s_2, t_E, t_D, t_I, t_O \}$. Of which, the $V$-terms are $\trees_V = \{ s_0, s_1, s_2 \}$. Which means that the vertices of $\omega_\bd(\rel B)$ are triples $(U_{s_0}, U_{s_1}, U_{s_2})$ where $U_{s_0} = \{\emptyset\}$, and $U_{s_1}, U_{s_2}$ are sets of functions from a 1-element set to $B$. We identify such a triple with a pair $(U^+, U^-)$ where $U^+\subseteq B$ is the set of images of function in $U_{s_1}$ and $U^-\subseteq B$ the set of images of functions in $U_{s_2}$.
Using a similar simplification of the definition of edges as in the previous example, we get that two such pairs $(U^+, U^-)$ and $(V^+, V^-)$ are connected by an edge if there exists $e_\bullet \in B$, so that the sets
\begin{align*}
  \e_{t_E}(U, V; e_\bullet) &= \{ e_\bullet \} &
  \e_{t_D}(U, V; e_\bullet) &= U^- \times \{ e_\bullet \} \\
  \e_{t_I}(U, V; e_\bullet) &= \{ e_\bullet \} \times V^- &
  \e_{t_O}(U, V; e_\bullet) &= U^+ \times \{ e_\bullet \}
\end{align*}
satisfy (B1) $\e_{t_E}(U, V; e_\bullet) \subseteq U^+ \cap V^-$ and (B2) $\e_{t_R}(U, V; e_\bullet) \subseteq R^{\rel B}$ for each $R \in \{D, I, O\}$.
This means that $(U^+, U^-)$ and $(V^+, V^-)$ are connected by an edge if there exists $e_\bullet \in U^+ \cap V^-$ such that
\[
  U^- \times \{ e_\bullet \} \subseteq D^{\rel B},
  \{ e_\bullet \} \times V^- \subseteq I^{\rel B},
  U^+ \times \{ e_\bullet \} \subseteq O^{\rel B}.
\]
This completes the definition, although we can further refine it by requiring that for each vertex, the sets $U^+$ and $U^-$ satisfy additional properties that would automatically imply the conditions above. Namely, we require
\[
  U^- \times U^+ \subseteq D^{\rel B},
  U^- \times U^- \subseteq I^{\rel B},
  U^+ \times U^+ \subseteq O^{\rel B}.
\]
To sum up the refined definition, we let $\omega_\bd(\rel B)$ be the digraph with vertex set
\[
  \{ (U^+, U^-) \mid
  U^- \times U^+ \subseteq D^{\rel B},
  U^- \times U^- \subseteq I^{\rel B},
  U^+ \times U^+ \subseteq O^{\rel B} \}
\]
where $(U^+, U^-)$ and $(V^+, V^-)$ form an edge if $U^+ \cap V^- \neq \emptyset$. It is not hard to check that even after the refinements, we still get the right adjoint to $\bd$, i.e., that indeed there is a homomorphism $\rel A \to \omega_\bd(\rel B)$ if and only if there is a homomorphism $\bd(\rel A) \to \rel B$ for any digraph $\rel A$ and $\phi$-structure $\rel B$. Note how this compares to the right adjoint $\delta_R$ of the arc digraph functor as defined in the previous subsection.

\subsection{Proof of Theorem \ref{thm:adjoint-edge}}

As noted above, the proof closely follows the proof of Theorem~\ref{thm:adjoint-1}, and we present it with the same structure starting with the easier of the two implications.

The following lemma is an analogue to Lemma~\ref{lem:A-easy}. The statement is identical, although the meaning of $\Gamma$ and $\Omega$ has changed. In the proof, $h$ is no longer a homomorphism from $\Gamma(\relT t) \to \Gamma(\rel A)$, hence we replace every instance of $h$ being used as such a homomorphism with $h^{\check S}$.  Furthermore, to show that $g$ preserves an edge $e \in {\check S}^\rel A$, we need to provide a witness, which is $w_\bullet = f(e)$.

\begin{lemma} \label{lem:general-easy}
  If there is a homomorphism $f\colon \Gamma(\rel A) \to \rel B$, then there is a homomorphism $g\colon \rel A \to \Omega(\rel B)$.
\end{lemma}

\begin{proof}
  Assuming $f\colon \Gamma(\rel A) \to \rel B$, we define a mapping $g\colon A \to \Omega(B)$ by setting, for all $u\in A$ and all $t\in \trees_V$,
  \[
    g(u)_t = \{ f\circ h^{\check S} \mid h\colon \relT t \to \rel A, h(r_t) = u \},
  \]
  and claim that it is a homomorphism from $\rel A$ to $\Omega(\rel B)$.
  We need to check that $g(u)$ is well-defined, i.e., that $f\circ h^{\check S}$ is a homomorphism from $\Gamma(\relT t)$ to $\rel B$ for all $t\in \trees_V$ and that $g(u)_{\vertex} = \{\emptyset\}$. Indeed, $f\circ h^{\check S}$ is a homomorphism since it is a composition of homomorphisms $h^{\check S}\colon \Gamma(\relT t) \to \Gamma(\rel A)$ and $f\colon \Gamma(\rel A) \to \rel B$. And, $g(u)_{\vertex} = \{\emptyset\}$ because $\relT{\vertex}$ has no ${\check S}$-edges, which implies that the unique mapping $h\colon T(\vertex) \to B$ such that $h(r_\vertex) = u$ is a homomorphism, and, moreover, this $h$ satisfies $h^{\check S} = \emptyset$. Consequently, $f\circ h^{\check S} = \emptyset$, as we wanted to show.

  To show that $g$ is a homomorphism, assume first that $S\neq {\check S}$ is a $\sigma$-symbol of arity $k$, and $e = (u_1, \dots, u_k) \in S^{\rel A}$. We claim that, for each $t\in \trees_S$,
  \[
    \e_t(g(e)) = \{ f\circ h^{\check S}
      \mid h\colon \relT{t} \to \rel{A}, h^S(r_t) = e \}.
  \]
  Let $t = \edge_S(t_1, \dots, t_k)$, and observe that homomorphisms $h\colon \relT t \to \rel A$ such that $h^S(r_t) = e$ are in 1-to-1 correspondence with $k$-tuples of homomorphisms $h_1, \dots, h_k$, such that $h_i \colon \relT{t_i} \to \rel A$ and $h_i(r_{t_i}) = u_i$ for all $i\in [k]$, obtained as their restrictions to the respective subtrees. If $h$ is the union of $h_i$'s then also $h^{\check S}$ is the union of $h_i^{\check S}$'s. The claim then easily follows.
  For (A1), we want to check that
  \(
    \e_t(g(e)) \subseteq g(u_i)_s
  \)
  where $s = \pr_i(t)$. Observe that if $h\colon \relT t \to \rel A$ and $h^S(r_t) = e$, then $h(r_s) = u_i$, since the $i$-th component of $r_t$ is $r_s$. The condition then follows from the claim.
  Finally, (A2) follows directly from the claim by the same argument as $g(u)_t \subseteq \hom(\relT t, \rel B)$.

  Second, for the case $S = {\check S}$ and $e = (u_1, \dots, u_k) \in \check S^{\rel A}$, we need to pick an element $w_\bullet \in B$ witnessing that $g^{\check S}(e) \in \check S^{\Omega(\rel B)}$ --- we pick $w_\bullet = f(e)$. We claim that, for each $t\in \trees_{\check S}$,
  \[
    \e_t(g^{\check S}(e); f(e)) = \{ f \circ h^{\check S} \mid h\colon \relT t \to \rel A, h^{\check S}(r_t) = e \}.
  \]
  The inclusion `$\supseteq$' is clear, since the restriction of $h$ to $T(t_i)$ is a homomorphism $h_i\colon \relT{t_i}\to \rel A$ such that $h_i(r_{t_i}) = u_i$. The other inclusion also follows, since, given homomorphisms $h_i\colon \relT{t_i}\to \rel A$ with $h_i(r_{t_i}) = u_i$, their union is a homomorphism $h\colon \relT{t} \to \rel A$ with $h^{\check S}(r_t) = e$. Moreover,
  \[
    h^{\check S} = h_1^{\check S} \cup \dots \cup h_k^{\check S} \cup (r_t \mapsto e),
  \]
  and hence
  \[
    f\circ h^{\check S} = (f\circ h_1^{\check S}) \cup \dots \cup (f\circ h_k^{\check S}) \cup (r_t \mapsto f(e)).
  \]
  The conditions (B1) and (B2) then follow from the claim similarly as above.
\end{proof}

The other implication is proved in the following two lemmas.
As in the case of Lemma~\ref{lem:A-hard-pre}, the first lemma provides the adjunction in the special case $\rel A = \relT t$ for $t\in \trees$, and we derive the general case, which is covered by the second lemma, using the special case.

The statement of Lemma~\ref{lem:A-hard-pre} is changed to account for differences in the definitions. The mapping $d$ is now defined to map $s$ to a witness of the edge $h^{\check S}(s)$ instead of the unique image of the element of $h(v)$. The proof is analogous, by induction on the term $t$, although now we have to distinguish four cases.

\begin{lemma}
  Let $t\in \trees$, $h\colon \relT t \to \Omega(\rel B)$, and $d\colon \check S^{\relT t} \to B$ be a map such that, for all $s\in \check S^{\relT t}$, $d(s)$ is a witness for the edge $h^{\check S}(s)$.
  Then $d$ is a homomorphism $\Gamma(\relT t) \to \rel B$.
\end{lemma}

\begin{proof}
  The proof is similar to the proof of Lemma~\ref{lem:A-hard-pre}. We first show by induction on the term $t$ that either $d \in h(r_t)_t$ if $t$ is a $V$-term, or $d \in \e_t(h^{\check S}(r_t); d(r_t))$ if $t$ is an $\check S$-term, or $d \in \e_t(h^S(r_t))$ if $t$ is an $S$-term for a symbol $S\ne {\check S}$.
  \begin{description}
    \item [Case $t = \vertex$]
      Since $\check S^{\relT\vertex} = \emptyset$, the only map $d \colon \check S^{\relT\vertex} \to B$ is the empty map, which is in $h(r_\vertex)_\vertex$ by definition.
    \item [Case $t = \edge_{\check S} (t_1, \dots, t_k)$]
      Note that restrictions of $h$ to subtrees $\relT{t_i}$'s are homomorphisms, hence, for all $i \in \{1, \dots, k\}$, $h(r_{t_i})_{t_i}$ contains the corresponding restrictions of $d$ by the inductive assumption. The claim then follows from the definition of $\e_t(h^{\check S}(r_t); d(r_t))$.
    \item [Case $t = \edge_S (t_1, \dots, t_k)$ where $S\neq \check S$]
      This is proved in the same way as the above case, with the exception that we use the definition of $\e_t(h^S(r_t))$ instead of $\e_t(h^{\check S}(r_t); d(r_t))$.
    \item [Case $t = \pr_i (s)$]
      Since $h$ is a homomorphism from $\relT t = \relT s$ to $\rel B$, we know that either $\e_s(h^S(r_s))$, if $s$ is an $S$-term for $S\neq \check S$, or $\e_s(h^{\check S}(r_s); d(r_s))$, if $s$ is an $\check S$-term, contains $d$ by the inductive assumption. The claim then follows from either (A1), or (B1).
  \end{description}
  The lemma then immediately follows by the definition if $t$ is a $V$-term, by (B2) if $t$ is an $\check S$-term, or by (A2) otherwise.
\end{proof}

Finally, in the proof of Lemma~\ref{lem:A-hard}, we define $f$ to map $e \in {\check S}^{\rel A}$ to a witness of the edge $g(e) \in {\check S}^{\Omega(\rel B)}$ instead of letting $f(u)$ be the unique value attained by the map in $g(u)_\vertex$. The rest of the proof is then a straightforward application of the above, i.e., completely analogous to the proof of Lemma~\ref{lem:A-hard} using Lemma~\ref{lem:A-hard-pre}.

\begin{lemma} \label{lem:general-hard}
  If there is a homomorphism $g\colon \rel A \to \Omega(\rel B)$, then there is a homomorphism $f\colon \Gamma(\rel A) \to \rel B$.
\end{lemma}

\begin{proof}
  As before, we assume that, for each $\tau$-symbol $R$, $\rel Q_R$ and $\relT{t_R}$ are equal and not just isomorphic.
  We define a mapping $f\colon \Gamma(A) \to B$ by setting, for all $s\in \check S^{\rel A}$, $f(s) = e_\bullet$ for some witness $e_\bullet$ of the edge $g^{\check S}(s) \in \check S^{\Omega(\rel B)}$, and claim that this $f$ is a homomorphism from $\Gamma(\rel A)$ to $\rel B$.

  We need to show that $f$ preserves each relation $R$. Assume that $(u_1,\dots, u_k) \in R^{\Gamma(\rel A)}$, i.e., there is a homomorphism $h\colon \rel Q_R \to \rel A$ such that
  \[
    (h^{\check S}(x_1), \dots, h^{\check S}(x_k)) = (u_1, \dots, u_k).
  \]
  The previous lemma applied to the homomorphism $g\circ h\colon \relT{t_R} \to \Omega(\rel B)$ in place of $h$ and the map $f\circ h^{\check S}$ in place of $d$ implies that $f\circ h^{\check S} \colon \Gamma(\rel Q_R) \to \rel B$ is a homomorphism, which in turn implies
  \[
    (f(u_1), \dots, f(u_k)) = (fh^{\check S}(x_1), \dots, fh^{\check S}(x_k)) \in R^{\rel B}
  \]
  since $(x_1, \dots, x_k) \in R^{\Gamma(\rel Q_R)}$, as we wanted to show.
\end{proof}

This concludes the proof of Theorem~\ref{thm:adjoint-edge}.
 
\section{Composition of adjoints}
  \label{sec:composition}

In this section, we give an example of what can be achieved by composing functors defined in Sections 5 and 6. The power of composing two adjoints to obtain more complicated constructions was observed in \cite[Section 5]{FT15}, where the authors considered composition of digraph functors with adjoints. This section gives several examples that show that we can obtain adjoints to more digraph functors by composing functors that go outside of the scope of digraphs into general relational structures. Naturally, our constructions also give more adjoints between general relational structures.

We start with a few general observations. The key fact that makes composition of adjoints useful is the following well-known category-theoretical observation.

\begin{lemma} \label{lem:composition}
  Assume that $\Lambda_1, \Gamma_1$ and $\Lambda_2, \Gamma_2$ are two pairs of (thin) adjoint functors, then $\Lambda_1\circ \Lambda_2$ is a left adjoint to $\Gamma_2\circ \Gamma_1$.
\end{lemma}

\begin{proof}
  The proof is straightforward. We get the following string of equivalences from the two adjoints: $\Lambda_1\Lambda_2(\rel A) \to \rel B$ if and only if $\Lambda_2(\rel A) \to \Gamma_1(\rel B)$ if and only if $\rel A \to \Gamma_2\Gamma_1(\rel B)$ for any two structures $\rel A$ and $\rel B$ of the right signatures.
\end{proof}

We note that a composition of Pultr functors gives a Pultr functor \cite{Pul70}. We include a sketch of a proof of a slightly weaker statement.

\begin{lemma}
  Let $\Lambda_1, \Gamma_1$ and $\Lambda_2, \Gamma_2$ be two pairs of left and central Pultr functors such that $\Lambda_1 \circ \Lambda_2$ and $\Gamma_2 \circ \Gamma_1$ are well-defined. Then there is a pair of Pultr functors $\Lambda$ and $\Gamma$ such that $\Lambda(\rel A)$ and $\Lambda_1 \circ \Lambda_2(\rel A)$ are isomorphic for all $\rel A$, and $\Gamma(\rel B)$ and $\Gamma_2 \circ \Gamma_1(\rel B)$ are homomorphically equivalent for all $\rel B$.
\end{lemma}

\begin{proof}[Proof sketch]
  The template of the composition can be obtained as the $\Lambda_1$-image of the Pultr template defining $\Lambda_2$ and $\Gamma_2$, i.e., if the template of $\Lambda_2$ and $\Gamma_2$ is composed of structures $\rel P$ and $\rel Q_R$, and homomorphisms $\epsilon_{i, R}$, then the template of the composition consists of structures $\Lambda_1(\rel P)$ and $\Lambda_1(\rel Q_R)$, and homomorphisms $\epsilon_{i, R}^{\Lambda_1}\colon \Lambda_1(\rel P) \to \Lambda_1(\rel Q_R)$ that are induced by $\Lambda_1$ from $\epsilon_{i, R}\colon \rel P \to \rel Q_R$.

  The composition $\Lambda_1 \circ \Lambda_2$ is, by defininion, a two-step process: In the first step (applying $\Lambda_2$), we replace each vertex with a copy of $\rel P$, and in the second step (applying $\Lambda_1$), each of these copies is replaced with a copy of $\Lambda_1(\rel P)$. Analogously, an $R$-edge is replaced with a copy of $\Lambda_1(\rel Q_R)$. It is also not hard to observe that the identification corresponds to the maps $\epsilon_{i, R}^{\Lambda_1}$, which concludes that, for all $\rel A$, $\Lambda_1 \circ \Lambda_2(\rel A)$ is isomorphic to $\Lambda(\rel A)$. We get that $\Gamma_2 \circ \Gamma_1(\rel B)$ and $\Gamma(\rel B)$ are homomorphically equivalent for all $\rel B$, since both functors are right adjoints to $\Lambda$.
\end{proof}

The following theorem is obtained by composing adjoints constructed in Sections~\ref{sec:adjoint-1} and \ref{sec:adjoint-edge}. Although it is not an exhaustive list of adjunctions that can be constructed by such compositions, it provides more adjoints to digraph functors on top of those provided in \cite{FT15}. The theorem concerns a relatively general case of Pultr templates where $\rel P$ is an arbitrary tree. We only require that copies of $\rel P$ in the respective $\rel Q_R$'s intersect in at most one vertex. We also note that this theorem covers the cases of central Pultr functors whose adjoints are provided by Theorems~\ref{thm:adjoint-1} and \ref{thm:adjoint-edge}.

\begin{theorem} \label{thm:7.1}
  Assume a~$(\sigma, \tau)$-Pultr template with $\rel P$ and all $\rel Q_R$'s being $\sigma$-trees such that, for each $\tau$-symbol $R$,
  \begin{enumerate}
    \item  $\epsilon_{i,R}$ is injective for all $i\in \{1, \dots, \ar R\}$, and
    \item for each $i\neq j$, $i,j\in \{1, \dots, \ar R\}$, the images of $\rel P$ under $\epsilon_{i,R}$ and $\epsilon_{j,R}$ intersect in at most one vertex.
  \end{enumerate}
  Then the corresponding central Pultr functor $\Gamma$ has a right adjoint.
\end{theorem}

\begin{proof}
  Assume that $P = \{1, \dots, p\}$. The goal is to decompose the functor $\Gamma$ into two Pultr functors $\Gamma_1$ and $\Gamma_2$. The intermediate step is to construct a structure of a new signature. This new signature $\upsilon$ is obtained from $\sigma$ by adding a new relational symbol $S$ of arity $p$ (the size of the domain of $\rel P$) while retaining all symbols in $\sigma$.

  We define the first functor, $\Gamma_1$ that maps $\sigma$-structures to $\upsilon$-structures. Essentially, this functor simply adds a new relation $S$ that is defined by $\rel P$, i.e., $\Gamma_1(\rel B)$ is the $\upsilon$-structure with domain $B$, where the relations are defined as
  \begin{align*}
      R^{\Gamma_1(\rel B)} &= R^{\rel B} \text{ for each $\sigma$-symbol $R$, and}\\
      S^{\Gamma_1(\rel B)} &= \{ (h(1),\dots,h(p)) \mid h \colon \rel P \to \rel B \}.
  \end{align*}
  It is clear that this functor is defined by a $(\sigma, \upsilon)$-Pultr template that satisfies the assumptions of Theorem~\ref{thm:adjoint-1}.

  The second functor, $\Gamma_2$ is defined by altering the original Pultr template for $\Gamma$. The new template consists of $\rel P'$ and $\rel Q'_R$'s defined as follows. First, $\rel P' = \rel S_1$. For each $\tau$-symbol $R$, we obtain $\rel Q'_R$ from $\rel Q_R$ by the following procedure: For each $i \in \{1, \dots, \ar R\}$, remove from $\rel Q_R$ all edges in the image of $\epsilon_{i, R}$, and add the edge $(\epsilon_{i,R}(1), \dots, \epsilon_{i,R}(p)) \in S^{\rel Q'_R}$.
  Finally, we let $\epsilon_{i,R}' = \epsilon_{i,R}$, for all $i$ and $R$, which is a homomorphism $\rel P' \to \rel Q'_R$ since we added the corresponding $S$-edge into $\rel Q'_R$.
  Observe that, since $\epsilon_{i,R}$'s are injective, the above procedure does not introduce reflexive tuples (i.e., tuples with repeated entries), and, since the images of $\epsilon_{i,R}$ and $\epsilon_{j,R}$ intersect in at most one vertex for $i\neq j$, this does not introduce cycles into $\rel Q'_R$ using $S$-edges, and, since all other edges have been removed in the image, it does not introduce any other cycles. Hence, $\rel Q'_R$ is still a tree for each $\tau$-symbol $R$, and therefore Theorem~\ref{thm:adjoint-edge} applies.

  The above two paragraphs show that $\Gamma_1$ and $\Gamma_2$ have adjoints $\Omega_1$ and $\Omega_2$. And it is straightforward to check that $\Gamma = \Gamma_2 \circ \Gamma_1$. This concludes that $\Omega_1 \circ \Omega_2$ is a right adjoint to $\Gamma$ by Lemma~\ref{lem:composition}.
\end{proof}

\section{Conclusion}
  \label{sec:conclusion}

We have studied the problem of characterising central Pultr functors for arbitrary relational structures that admit a right adjoint, and, for those that do, giving an explicit construction for such an adjoint.
There is a necessary condition for the existence of such an adjoint (cf.\ Theorem~\ref{thm:adjoints-and-duals} and comments after it). We gave a sufficient condition in Theorem~\ref{thm:7.1}. These two conditions do not match, there is a gap between them, and it is not quite clear what the necessary and sufficient condition should be (even in the case of digraphs). Apart from the requirement that $\rel P$ and all $\rel Q_R$'s are trees, Theorem \ref{thm:7.1} has two additional assumptions. We believe that the second assumption (about intersection of images of $\rel P$ in $\rel Q_R$) is a technicality that can be removed with some extra work. How essential is the first assumption (about injectivity of homomorphisms $\epsilon_{i,R}$)? For example, is it true that, for every central Pultr functor $\Gamma$ that has a right adjoint, there is another central Pultr functor $\Gamma'$ such that (a) for every structure $\rel A$ of appropriate signature, $\Gamma(\rel A)$ and $\Gamma'(\rel A)$ are homomorphically equivalent, and (b) the Pultr template corresponding to $\Gamma'$ has all homomorphisms $\epsilon_{i,R}$ injective?

Finally, let us discuss possible applications of our results in the complexity of (promise) CSPs.
Left and central Pultr functors can alternatively be described as ``gadget replacements'' and ``pp-constructions'', respectively. These two constructions are central to the algebraic theory of complexity of (promise) CSPs where gadget replacements are used as log-space reductions between (promise) CSPs whose templates are related via pp-constructions; see, e.g., \cite{BBKO21} or \cite[Section 4.1]{KOWZ20}.
The opposite type of reductions, where gadgets are replaced by individual constraints, is less common, but has already been shown to be important; see, e.g., \cite{KOWZ20,BK22}.
It was discovered recently, that the right adjoints to some functor can be used to characterise when the functor is a valid reduction between two (promise) CSPs \cite[Theorem 4.6]{KOWZ20}.
Moreover, recent developments in the theory of promise CSPs have called for characterisation of reductions which include central Pultr functors \cite{BK22,KO22,DO23}. Our results provide a characterisation of when a reduction from a general class is a valid reduction between two promise CSPs. Although we are not aware of a specific interesting case of a promise CSP whose hardness is proven by providing a right adjoint to a central Pultr functor, apart from the use of the arc graph construction in \cite{WZ20}, our general results can be used as a foundation for further investigation of these reductions.
 
\subsection*{Acknowledgements}

We would like to thank the reviewers for many insightful comments on the manuscript.

\bibliographystyle{alphaurl}

\begin{thebibliography}{KOW{\v{Z}}23}

\bibitem[AGH17]{AGH17}
Per Austrin, Venkatesan Guruswami, and Johan H{\aa}stad.
\newblock (2+{$\epsilon$})-{S}at is {NP}-hard.
\newblock {\em {SIAM} J. Comput.}, 46(5):1554--1573, 2017.
\newblock \url{https://eccc.weizmann.ac.il/report/2013/159/}, \href
  {https://doi.org/10.1137/15M1006507} {doi:10.1137/15M1006507}.

\bibitem[BBKO21]{BBKO21}
Libor Barto, Jakub Bul\'{\i}n, Andrei Krokhin, and Jakub Opr\v{s}al.
\newblock Algebraic approach to promise constraint satisfaction.
\newblock {\em J. ACM}, 68(4):28:1--66, August 2021.
\newblock \href {https://doi.org/10.1145/3457606} {doi:10.1145/3457606}.

\bibitem[BK22]{BK22}
Libor Barto and Marcin Kozik.
\newblock Combinatorial gap theorem and reductions between promise {CSP}s.
\newblock In {\em Proceedings of the 2022 Annual {ACM}-{SIAM} Symposium on
  Discrete Algorithms ({SODA})}, pages 1204--1220, Virtual Conference, 2022.
  SIAM.
\newblock \href {https://doi.org/10.1137/1.9781611977073.50}
  {doi:10.1137/1.9781611977073.50}.

\bibitem[BKW17]{BKW17}
Libor Barto, Andrei Krokhin, and Ross Willard.
\newblock Polymorphisms, and how to use them.
\newblock In Andrei Krokhin and Stanislav {\v Z}ivn{\' y}, editors, {\em The
  Constraint Satisfaction Problem: Complexity and Approximability}, volume~7 of
  {\em Dagstuhl Follow-Ups}, pages 1--44. Schloss Dagstuhl -- Leibniz-Zentrum
  f{\"u}r Informatik, Dagstuhl, Germany, 2017.
\newblock \url{https://drops.dagstuhl.de/opus/frontdoor.php?source_opus=6959},
  \href {https://doi.org/10.4230/DFU.Vol7.15301.1}
  {doi:10.4230/DFU.Vol7.15301.1}.

\bibitem[DO23]{DO23}
Victor Dalmau and Jakub Opr\v{s}al.
\newblock Local consistency as a reduction between constraint satisfaction
  problems.
\newblock {\em preprint}, 2023.
\newblock \href {http://arxiv.org/abs/2301.05084} {arXiv:2301.05084}, \href
  {https://doi.org/10.48550/arXiv.2301.05084} {doi:10.48550/arXiv.2301.05084}.

\bibitem[FT13]{FT13}
Jan Foniok and Claude Tardif.
\newblock Adjoint functors in graph theory.
\newblock {\em preprint}, 2013.
\newblock \href {http://arxiv.org/abs/1304.2215} {arXiv:1304.2215}.

\bibitem[FT15]{FT15}
Jan Foniok and Claude Tardif.
\newblock Digraph functors which admit both left and right adjoints.
\newblock {\em Discrete Math.}, 338(4):527--535, 2015.
\newblock \href {http://arxiv.org/abs/1304.2204} {arXiv:1304.2204}, \href
  {https://doi.org/10.1016/j.disc.2014.10.018}
  {doi:10.1016/j.disc.2014.10.018}.

\bibitem[FV98]{FV98}
Tom{\'{a}}s Feder and Moshe~Y. Vardi.
\newblock The computational structure of monotone monadic {SNP} and constraint
  satisfaction: A study through datalog and group theory.
\newblock {\em SIAM J. Comput.}, 28(1):57--104, February 1998.
\newblock \href {https://doi.org/10.1137/S0097539794266766}
  {doi:10.1137/S0097539794266766}.

\bibitem[HN04]{HN04}
Pavol Hell and Jaroslav Ne\v{s}et\v{r}il.
\newblock {\em Graphs and Homomorphisms}.
\newblock Oxford University Press, 2004.

\bibitem[KO22]{KO22}
Andrei Krokhin and Jakub Opr{\v{s}}al.
\newblock An invitation to the promise constraint satisfaction problem.
\newblock {\em ACM SIGLOG News}, 9(3):30--59, 2022.
\newblock \href {http://arxiv.org/abs/2208.13538} {arXiv:2208.13538}, \href
  {https://doi.org/10.1145/3559736.3559740} {doi:10.1145/3559736.3559740}.

\bibitem[KOW{\v{Z}}23]{KOWZ20}
Andrei Krokhin, Jakub Opr{\v{s}}al, Marcin Wrochna, and Stanislav
  {\v{Z}}ivn{\'y}.
\newblock Topology and adjunction in promise constraint satisfaction.
\newblock {\em SIAM J. Comp.}, 52(1):38--79, 2023.
\newblock \href {http://arxiv.org/abs/2003.11351} {arXiv:2003.11351}, \href
  {https://doi.org/10.1137/20M1378223} {doi:10.1137/20M1378223}.

\bibitem[K{\v{Z}}17]{KZ17book}
Andrei Krokhin and Stanislav {\v{Z}}ivn{\'{y}}, editors.
\newblock {\em The Constraint Satisfaction Problem: Complexity and
  Approximability}, volume~7 of {\em Dagstuhl Follow-Ups}.
\newblock Schloss Dagstuhl - Leibniz-Zentrum f{\"{u}}r Informatik, 2017.
\newblock \url{http://www.dagstuhl.de/dagpub/978-3-95977-003-3}.

\bibitem[LLT07]{LLT07}
Beno{\^{\i}}t Larose, Cynthia Loten, and Claude Tardif.
\newblock A characterisation of first-order constraint satisfaction problems.
\newblock {\em Logical Methods in Computer Science}, 3(4), 2007.
\newblock \href {https://doi.org/10.2168/LMCS-3(4:6)2007}
  {doi:10.2168/LMCS-3(4:6)2007}.

\bibitem[NT00]{NT00}
Jaroslav Nešetřil and Claude Tardif.
\newblock Duality theorems for finite structures (characterising gaps and good
  characterisations).
\newblock {\em Journal of Combinatorial Theory, Series B}, 80(1):80--97, 2000.
\newblock \href {https://doi.org/10.1006/jctb.2000.1970}
  {doi:10.1006/jctb.2000.1970}.

\bibitem[NT05]{NT05}
Jaroslav Ne{\v{s}}et{\v{r}}il and Claude Tardif.
\newblock Short answers to exponentially long questions: Extremal aspects of
  homomorphism duality.
\newblock {\em SIAM J. Discret. Math.}, 19(4):914--920, August 2005.
\newblock \href {https://doi.org/10.1137/S0895480104445630}
  {doi:10.1137/S0895480104445630}.

\bibitem[Pul70]{Pul70}
Ale{\v{s}} Pultr.
\newblock The right adjoints into the categories of relational systems.
\newblock In {\em Reports of the Midwest Category Seminar IV}, volume 137 of
  {\em Lecture Notes in Mathematics}, pages 100--113. Springer, 1970.
\newblock \url{http://www.springer.com/us/book/9783540049265}, \href
  {https://doi.org/10.1007/BFb0060437} {doi:10.1007/BFb0060437}.

\bibitem[W{\v{Z}}20]{WZ20}
Marcin Wrochna and Stanislav {\v{Z}}ivn{\'y}.
\newblock Improved hardness for {$H$}-colourings of {$G$}-colourable graphs.
\newblock In {\em Proceedings of the Fourteenth Annual ACM-SIAM Symposium on
  Discrete Algorithms (SODA'20)}, pages 1426--1435, 2020.
\newblock \href {http://arxiv.org/abs/1907.00872} {arXiv:1907.00872}, \href
  {https://doi.org/10.1137/1.9781611975994.86}
  {doi:10.1137/1.9781611975994.86}.

\end{thebibliography}

\end{document}